\documentclass[12pt]{amsart}

\usepackage{fullpage}
\usepackage{latexsym}
\usepackage{amssymb}
\usepackage{pinlabel}
\usepackage[new]{old-arrows}
\usepackage[all]{xy}

\usepackage[utf8]{inputenc}

\newtheorem{thm}{Theorem}[section]
\newtheorem{prop}[thm]{Proposition}
\newtheorem{cor}[thm]{Corollary}
\newtheorem{con}[thm]{Conjecture}
\newtheorem{lem}[thm]{Lemma}
\newtheorem{que}[thm]{Question}

\theoremstyle{definition}
\newtheorem{rem}[thm]{Remark}
\newtheorem{defn}[thm]{Definition}
\newtheorem{ex}[thm]{Example}

\usepackage{hyperref}

\newcommand{\N}{\mathbb{N}}

\newcommand{\R}{\mathbb{R}}
\newcommand{\Z}{\mathbb{Z}}

\newcommand\rank{\operatorname{rank}}

\title[Mapping tori, Hopf-type properties, rigidity]
{Endomorphisms of mapping tori}
\author{Christoforos Neofytidis}
\address{Department of Mathematics, Ohio State University, Columbus, OH 43210, USA}
\email{neofytidis.1@osu.edu}
\date{\today}

\begin{document}

\maketitle

\begin{abstract}
We classify in terms of Hopf-type properties mapping tori of residually finite Poincar\'e Duality groups with non-zero Euler characteristic.
This generalises and gives a new proof of the analogous classification for fibered 3-manifolds. 
Various applications are given. In particular, we deduce that  rigidity results for 
Gromov hyperbolic groups hold for the above mapping tori with trivial center.
\end{abstract}

\section{Introduction}

We classify in terms of Hopf-type properties mapping tori of 
residually finite Poincar\'e Duality ($PD$) groups $K$ with non-zero Euler characteristic, which satisfy the following finiteness condition:
\begin{equation}\tag{$\ast$}\label{finiteness}
\begin{aligned}[c]
\text{\em Let} \ \theta\in\mathrm{Out}(K). \ \text{\em If} \  \theta^d=\xi\theta\xi^{-1} \ \text{\em for some} \ d>1 \ \text{\em and} \ \xi\in\mathrm{Out}(K), \ \text{\em then} \ \theta \ \text{\em is torsion}.\\
  \end{aligned}
\end{equation}
One sample application is that every endomorphism onto a finite index subgroup of the fundamental group of a mapping torus $F\rtimes_h S^1$, where $F$  is a closed aspherical manifold with $\pi_1(F)=K$ as above, induces a homotopy equivalence (equivalently, $\pi_1(F\rtimes_h S^1)$ is {\em cofinitely Hopfian}) if and only if 
the center of $\pi_1(F\rtimes_h S^1)$ is trivial; equivalently, $\pi_1(F\rtimes_h S^1)$ is {\em co-Hopfian}. If, in addition, $F$ is topologically rigid, then homotopy equivalence can be replaced by a map homotopic to a homeomorphism, by Waldhausen's rigidity~\cite{Wald} in dimension three and Bartels-L\"uck's rigidity~\cite{BL} in dimensions greater than four. Condition (\ref{finiteness}) is known to be true for every aspherical surface by a theorem of Farb, Lubotzky and Minsky~\cite{FLM} on translation lengths. Hence, our result generalises and gives a new, uniform proof of the analogous classification in dimension three, due to Gromov~\cite{Gro} for 3-manifolds with positive simplicial volume, and to Wang~\cite{Wan} for non-trivial graph 3-manifolds. 
We will not depend on any notion of hyperbolicity for manifolds or groups, or on non-vanishing 
semi-norms. In fact, as another application, we will show that there exist semi-norms that do not vanish on every such mapping torus $F\rtimes_h S^1$ with trivial center. 
Our results hold if we replace the requirement on the normal subgroup being a $PD^n$-group with weaker assumptions on co-homology of degree $n$. Furthermore, our methods
will generalise and encompass the analogous classification for free-by-cyclic groups which was done by Bridson, Groves, Hillman and Martin~\cite{BGHM}. 
Various other applications are given, for instance, to the Hopf problem, the Lichnerowicz problem on quasiregular maps, the Gromov order and topological rigidity,  as well as an extension to non-aspherical manifolds. 

\subsection{Motivation and statements of the main results}

Gromov~\cite{Gro} established strong rigidity for hyperbolic 3-manifolds and, more generally, for 3-manifolds with hyperbolic JSJ pieces, using the simplicial volume. This, together with Wang's study of self-maps of non-trivial graph 3-manifolds~\cite{Wan} and Waldhausen's rigidity~\cite{Wald}, yields the following strong rigidity for  3-manifolds
 that fiber over the circle, according to the Nielsen-Thurston picture.

\begin{thm}[Gromov, Waldhausen, Wang]\label{t:surfacebundles}
Let $E_h$ be the mapping torus of a diffeomorphism $h$ of a hyperbolic surface $\Sigma$. Every self-map of $E_h$ of non-zero degree is homotopic to a homeomorphism if and only if $h$ is not periodic.
\end{thm}

Our goal is to show that the hyperbolicity of the fiber $\Sigma$ or of its fundamental group $\pi_1(\Sigma)$, 
the (possible) hyperbolicity of the total space $E_h$, of one of the JSJ pieces of $E_h$ or of its fundamental group $\pi_1(E_h)$, and the non-vanishing of the simplicial volume of $E_h$ or of any other functorial semi-norm of $E_h$ (such as the Seifert volume) are 
unnecessary conditions to prove Theorem \ref{t:surfacebundles}. The proof is rather based on two properties of $\pi_1(\Sigma)$, namely on the {\em non-vanishing of the Euler characteristic} and on {\em residual finiteness}. We will generalise  Theorem \ref{t:surfacebundles} in every dimension  by removing the assumption on the fiber being a surface. 

 \begin{thm}\label{t:mainmanifolds}
Let $E_h$ be the mapping torus of a homeomorphism $h$ of a closed aspherical manifold $F$, which has non-zero Euler characteristic and residually finite fundamental group. If $\pi_1(F)$ satisfies condition {\normalfont(\ref{finiteness})}, then
 the following are equivalent:
\begin{itemize}
\item[(i)] $\pi_1(E_h)$ has trivial center $C(\pi_1(E_h))$;
\item[(ii)] Every endomorphism of $\pi_1(E_h)$ onto a finite index subgroup induces a 
homotopy equivalence on $E_h$;
\item[(iii)] Every injective endomorphism of $\pi_1(E_h)$ induces a homotopy equivalence on $E_h$;
\item[(iv)] Every self-map of $E_h$ of non-zero degree is a homotopy equivalence.
\end{itemize}
If, moreover, $F$ is topologically rigid (i.e., it satisfies the Borel conjecture), then homotopy equivalence in (ii), (iii) and (iv) can be replaced by a map homotopic to a homeomorphism.
\end{thm}

Our approach to Theorem \ref{t:mainmanifolds} will be purely algebraic 
and the involved groups need not be fundamental groups of closed aspherical manifolds. The key idea here is to characterise all mapping tori whose normal subgroup satisfies the above properties of $\pi_1(F)$, together with certain finiteness cohomological conditions. 

 \begin{thm}\label{t:maingroups}
Suppose $K$ is a 
residually finite group of type $FP$, such that $\chi(K)\neq0$, $H_n(K;\R)\neq0$ and $H^n(K;\Z K)$ is finitely generated, where $n$ is the cohomological dimension of $K$. If $K$ satisfies {\normalfont(\ref{finiteness})}, then the following are equivalent for the mapping torus $\Gamma_\theta=K\rtimes_\theta\Z$ of any automorphism $\theta\colon K\to K$:
\begin{itemize}
\item[(i)] $C(\Gamma_\theta)=1$;
\item[(ii)] $\Gamma_\theta$ is cofinitely Hopfian, i.e., every endomorphism of $\Gamma_\theta$ whose image is of finite index is an automorphism;
\item[(iii)] $\Gamma_\theta$ is co-Hopfian, i.e., every injective endomorphism of $\Gamma_\theta$ is an automorphism.
\end{itemize}
\end{thm}

It should be pointed out that condition (\ref{finiteness}) is not generally true once the finiteness cohomological conditions of Theorem \ref{t:maingroups} are relaxed; see  Example~\ref{ex:positiveEuler-infiniteorder1}. We will discuss condition (\ref{finiteness}) in more detail in Section \ref{ss:Lefschetz}. 

Recall that a free group $F_r$ on $r > 1$ generators is residually finite and has Euler characteristic $\chi(F_r) = 1- r\neq 0$. It moreover satisfies condition (\ref{finiteness}) by a result of Alibegovi\'c~\cite{Ali}, again on translation lengths. A result analogous  to Theorem \ref{t:maingroups}, for free-by-cyclic groups, was proved by Bridson, Groves, Hillman and Martin~\cite[Theorem B]{BGHM} (note that $H^1(F_r,\Z F_r)$ is not finitely generated). In contrast to our Theorem \ref{t:maingroups}, that result does not include the co-Hopf property (\cite[Remark 6.4]{BGHM}), but it includes another Hopf-type condition, called {\em hyper-Hopf} property. 
We will prove (Theorems \ref{t:Hopf-type} and \ref{t:finite-hyperHopf}) that all  of the equivalences (except the hyper-Hopf property) shown in~\cite{BGHM} for $F_r$-by-$\Z$ groups hold as well for $K$-by-$\Z$ groups, where $K$ is as in Theorem \ref{t:maingroups}. For surface-by-cyclic groups, we will prove that all Hopf-type properties are equivalent, including both the co-Hopf and the hyper-Hopf properties (Theorem \ref{t:surface}).

\subsection{Applications}

Theorems \ref{t:mainmanifolds} and \ref{t:maingroups} apply to a variety of problems in topology and geometry. The proofs of the results below will be given in Section \ref{s:applications} and they all reveal that our mapping tori with trivial center behave similarly to non-elementary Gromov hyperbolic groups.

\subsection*{The 3-manifold case}

Nielsen-Thurston's classification of surface automorphisms~\cite{Nie,Thu} tells us that an automorphism $h$ of a hyperbolic surface is either periodic, reducible or pseudo-Anosov. The mapping torus $E_h$ admits no self-maps of degree greater than one when $h$ is either pseudo-Anosov or reducible and  there is a hyperbolic JSJ piece in $E_h$, because in those cases $E_h$ has non-zero simplicial volume by Gromov's work~\cite{Gro}. Wang~\cite{Wan} showed that the remaining mapping tori of reducible surface automorphisms -- which are exactly those consisting only of Seifert JSJ pieces and thus their  simplicial volume vanishes -- do not admit self-maps of degree greater than one, by exploiting the hyperbolic orbifold bases of the Seifert JSJ pieces. Subsequently, Derbez and Wang~\cite{DW} showed that the latter 3-manifolds have virtually positive Seifert volume. 
Theorem \ref{t:mainmanifolds} gives a new, uniform proof in all cases, without using hyperbolicity or semi-norms.

\begin{thm}
\label{t:GromovWang}
Let $h$ be a diffeomorphism of a closed hyperbolic surface $\Sigma$. The mapping torus $E_h$ admits a self-map of degree greater than one if and only if $h$ is periodic.
\end{thm}

Combining Theorem \ref{t:GromovWang} (and its proof) with Waldhausen's rigidity for Haken 3-manifolds~\cite{Wald} we obtain an alternative proof for Theorem \ref{t:surfacebundles}.

\subsection*{Gromov-Thurston norms}

As motivated by the Thurston norm for 3-manifolds~\cite{Thu2}, the simplicial volume $\|M\|$ of a closed oriented $n$-dimensional manifold $M$ measures how efficiently the fundamental class of $M$ can be represented by real cycles~\cite{Gro}. An important property of the simplicial volume is the {\em functorial property}, which means that whenever there is a map $f\colon M\to N$, then  $\|M\|\geq|\deg(f)|\|N\|$. The existence and computation of functorial semi-norms in the sense of Gromov and Thurston is a hard problem~\cite{Gro2,MT}. As mentioned above, a 3-manifold $E_h$ given in Theorem \ref{t:surfacebundles}  admits a non-vanishing functorial semi-norm if and only if $h$ is not periodic~\cite{Gro,DW}. Theorem \ref{t:mainmanifolds}  generalises the 3-manifold case in every dimension. 

\begin{thm}\label{t:seminorms}
Let $E_h$ be as in Theorem \ref{t:mainmanifolds}. There is a functorial semi-norm that is non-zero and finite on $E_h$  if and only if $h_*$ has infinite order in $\mathrm{Out}(\pi_1(F))$. 
\end{thm}

\subsection*{The Hopf problem}

A long-standing question of Hopf asks whether every self-map $f$ of a closed manifold $M$ of degree $\pm1$  is a homotopy equivalence; this is Problem 5.26 in Kirby's list~\cite{Kir} (see also~\cite{Neum,Hau}). If $M$ is aspherical and $\pi_1(M)$ is Hopfian, then the answer is affirmative, since maps of degree $\pm1$ are $\pi_1$-surjective. The fundamental groups of our mapping tori $E_h$ are Hopfian being residually finite by Mal'cev's theorem~\cite{Mal}. More interestingly, all known examples of self-maps of aspherical manifolds of non-zero degree are either homotopic to a non-trivial covering or homotopic to  a homeomorphism when the degree is $\pm1$, as predicted by the Borel conjecture~\cite{BL,Bel,BHM,FJ,Gro,Gro1,Min,Min1,NeoHopf,Neoorder,Sel1,Sel2,Wald,Wan}. It is therefore natural to ask whether this is always true, strengthening Hopf's problem for the class of aspherical manifolds~\cite[Problem 1.2]{NeoHopf} (see Question \ref{q:Hopfstrong}). We obtain an affirmative answer for $E_h$.

\begin{thm}\label{t:Hopf}
Let $E_h$ be as in Theorem \ref{t:mainmanifolds}, such that $F$ is topologically rigid. Every self-map of $E_h$ of non-zero degree is homotopic either to a homeomorphism or to a
 non-trivial covering. In the latter case, $E_h$ is virtually trivial.
 \end{thm}

In the course of this study, we will also generalise in higher dimensions Wang's $\pi_1$-injectivity for self-maps of Seifert 3-manifolds that are modeled on the geometry $\mathbb{H}^2\times\R$~\cite{Wan}. 
We will use the residual finiteness of the fundamental group of the fiber, but not of the total space, that is, we will not use Hirshon's generalisation~\cite{Hir} of Mal'cev's theorem~\cite{Mal}; see Remark \ref{r:Wanggeneral} and Proposition \ref{p:Wang}. 

\subsection*{Gromov's ordering and rigidity}

Given two closed oriented manifolds $M$ and $N$ of the same dimension, we say that $M$ {\em dominates} $N$ if there is a map of non-zero degree from $M$ to $N$. Gromov suggested studying the domination relation as defining a partial ordering on the homotopy types of manifolds~\cite{CT}. In any dimension not equal to four, if $M$ and $N$  are negatively curved or, more generally, aspherical with hyperbolic fundamental groups, such that $M$ and $N$ dominate each other, then $M$ and $N$ are homeomorphic. This follows from profound results both at algebraic and geometric levels, such as Sela's study of Hopf-type properties of hyperbolic groups~\cite{Sel1,Sel2}, works by Gromov~\cite{Gro1,Gro2} and Mineyev~\cite{Min,Min1} on bounded cohomology and hyperbolic groups, and topological rigidity results by Farrell and Jones~\cite{FJ} and Bartels and L\"uck~\cite{BL}; concrete statements are given in~\cite{BHM} and in~\cite{NeoHopf}. Our results imply that the same topological rigidity holds for all non-virtually trivial mapping tori $E_{h}$.

\begin{thm}\label{t:Gromovorder}
For $i=1,2$, let $E_{h_i}$ be as in Theorem \ref{t:mainmanifolds}, such that $C(\pi_1(E_{h_i}))=1$. If $E_{h_1}$ and $E_{h_2}$ dominate each other, then they are homotopy equivalent. If,  moreover, the $F_i$ are topologically rigid, then  $E_{h_1}$ and $E_{h_2}$ are homeomorphic.
\end{thm}

\subsection*{The Lichnerowicz problem (quasiregular maps).}

Lelong-Ferrand's proof~\cite{Lel} of the Lichnerowicz conjecture~\cite{Lic} suggests to determine all closed  manifolds $M$  
which admit non-injective uniformly quasiregular maps, i.e., maps that preserve some bounded measurable conformal structure on $M$; this is known as the (generalised)  Lichnerowicz problem; cf.~\cite[p. 1614]{BHM} and~\cite[p. 2092]{MMP}. Bridson, Hinkkanen and Martin~\cite{BHM} showed that  closed aspherical manifolds with hyperbolic fundamental group do not admit any non-trivial quasiregular maps. Our Theorem \ref{t:maingroups} implies that the same is true for all mapping tori $E_h$ with trivial center.

\begin{thm}\label{t:quasiregular}
If $M$ is a closed manifold whose fundamental group is given as in Theorem \ref{t:maingroups} and has trivial center, then every quasiregular map $f\colon M\to M$ is a homeomorphism.
\end{thm}

Strong motivation for the study of quasiregular maps stems  from works of Walsh~\cite{Wal} and Smale~\cite{Sma} (see Lemma \ref{l:WalshSmale}), which establish a tight  connection between the cofinite Hopf property of groups and open maps of CW-complexes. 
Thus, Theorem \ref{t:maingroups} and its consequences apply to open maps between CW-complexes whose fundamental groups satisfy the conditions of Theorem \ref{t:maingroups} and have trivial center; see Remark~\ref{r:WS} and Theorem \ref{t:quasiregulargeneral}.

\subsection{Euler characteristic, $L^2$-Betti numbers, and Lefschetz numbers}

It is worth pointing out that the arguments in this paper that require the non-vanishing of the Euler characteristic can be carried out 
by assuming that {\em some}   $L^2$-Betti number of $K$ is not zero. While the latter is guaranteed if $\chi(K)\neq 0$ by the Euler characteristic formula with $L^2$-Betti numbers~\cite{Luebook}, it is not known whether the converse is always true for aspherical manifolds. 
The Singer and Hopf conjectures~\cite{Sin} (see also~\cite{Luebook}) predict that it is always true, asserting that for a closed aspherical manifold $M$ only the middle-dimensional $L^2$-Betti number might survive and $(-1)^{\dim(M)/2}\chi(M)\geq0$ when $\dim(M)$ is even. Hence, at least at a theoretical level, assuming that some  $L^2$-Betti number is not zero  gives a more general statement for  Theorem \ref{t:mainmanifolds} than assuming the non-vanishing of the Euler characteristic.

However, we prefer to state our results using the Euler characteristic instead of $L^2$-Betti numbers, because we can offer evidence that condition (\ref{finiteness}) is always true for aspherical manifolds (more generally, for groups that fulfil cohomological finiteness conditions similar to those of Theorem \ref{t:maingroups}), whenever the Euler characteristic is not zero.

\begin{thm}\label{t:finitenessEuler-Lefschetz}
Let $K$ be a group of finite homological type, $\mathrm{cd}(K)<\infty$, and $\chi(K)\neq0$. Suppose $\theta\colon K\to K$ is an automorphism which is conjugate to $\theta^d$ in $\mathrm{Out}(K)$ for some $d>1$. Then there exists an integer $q_m=q(\chi(K),d,m)>1$ such that $\chi(K)=\Lambda(\theta^{q_m})$, where $\Lambda(\cdot)$ denotes the Lefschetz number.
 \end{thm}

As we shall see in Example \ref{ex:positiveEuler-infiniteorder1}, the above theorem does not suffice to draw any conclusion on the finiteness of the order of $\theta$ in $\mathrm{Out}(K)$, and what fails in Example \ref{ex:positiveEuler-infiniteorder1} is the finite generation of $H^n(K;\Z K)$. We suspect that for a Poincar\'e Duality group $K$, the Euler characteristic of the fixed point subgroup of some conjugate of $\theta^{q_m}$ will be equal to $\Lambda(\theta^{q_m})$. Then, the same argument as for (part of) the proof of
Theorem \ref{t:co-Hopf} will imply the finiteness of the order of $\theta$ in $\mathrm{Out}(K)$, establishing condition (\ref{finiteness}) for $K$.

\subsection{A remark about hyperbolicity}
In~\cite{Gro}, Gromov reproved Mostow's rigidity~\cite{Mos} in every dimension greater than two, by showing and using the non-vanishing of the simplicial volume for hyperbolic manifolds. Mostow's rigidity implies that the fundamental groups of hyperbolic manifolds of dimensions greater than two have finite outer automorphism groups, which means that the corresponding $E_h$ are virtually trivial. In particular, mapping tori of hyperbolic manifolds of dimension greater than two have always vanishing simplicial volume.  The same is true for any mapping torus of a closed aspherical manifold of dimension at least three with hyperbolic fundamental group, by work of Paulin~\cite{Pau}, Bestvina and Feighn~\cite{BF} and Bowditch~\cite{Bow}. This indicates that, indeed, the hyperbolicity of $\pi_1(\Sigma)$ in Theorem \ref{t:surfacebundles} is not the key condition. Furthermore, for a group extension
\begin{equation}\label{eq:fiber}
1\longrightarrow K\longrightarrow\Gamma_\theta\longrightarrow\Z\longrightarrow1,
\end{equation}
the group $\Gamma_\theta$ being hyperbolic implies severe restrictions on the normal subgroup $K$. It is a folklore conjecture that there is no closed hyperbolic manifold of (obviously odd) dimension greater than three that fibers over the circle and, thus, its fundamental group has the form $\Gamma_\theta$ as in (\ref{eq:fiber}).
Hence, it is evident that nor the hyperbolicity of $\Gamma_\theta$ is a property that can be used for what we wish to do.  
From a purely group theoretic point of view, our results suggest to think of higher dimensional analogues of hyperbolic 3-manifold groups as
not being hyperbolic groups themselves, but rather members of a class of groups which have a finite index subgroup $\Gamma_\theta$ given by (\ref{eq:fiber}), where $K$ is torsion-free, residually finite with $\chi(K)\neq0$, and $\theta$ has infinite order in $\mathrm{Out}(K)$. Every hyperbolic 3-manifold group fits into this context being virtually fibered~\cite{Ago,Ago2}. From a geometric point of view, the above 
rigidity applications suggest an abundance of examples of non-positively curved manifolds $F$ which fit into the context of Theorem \ref{t:mainmanifolds}, e.g.,  {\em higher dimensional graph manifolds} studied in the monograph of Frigerio, Lafont and Sisto~\cite{FLS}. And, clearly, taking products whose factors consist, for instance, of such higher dimensional graph manifolds and of aspherical manifolds with hyperbolic fundamental groups will produce more examples. This goes beyond aspherical manifolds, as already indicated by the case of open maps between CW-complexes (Section \ref{ss:quasiregular}), or by considering, for example, connected sums and mixed products of aspherical manifold groups with free groups (Section \ref{s:non-aspherical}).

\subsection*{Acknowledgments}
I would like to thank Jean-Fran\c cois Lafont for useful discussions, Dennis Sullivan for his comments and encouragement, and an anonymous referee for the constructive feedback.

\section{Residual finiteness, injectivity, and the role of center}

\subsection{The structure of $\Gamma_\theta$}

Given a finite CW-complex  $X$ and a homeomorphism $h\colon X\to X$, the {\em mapping torus} of $h$ is defined to be
$$ E_h=X\rtimes_h S^1:=(X\times [0,1])/((x,0)\sim (h(x),1)).$$
$E_h$ fibers over the circle with natural projection $p\colon E_h\to S^1$ and fiber $X$. The  fundamental group of $E_h$ has a presentation 
\[
\pi_1(E)=\langle \pi_1(X), t \ | \ t\gamma t^{-1}=h_*(\gamma), \ \forall \gamma\in\pi_1(X)\rangle,
\]
where a chosen generator $t\in\mathbb{Z}$ acts on $\pi_1(X)$ by $\gamma\mapsto h_*(\gamma)=t\gamma t^{-1}, \ \gamma\in\pi_1(X)$, and there is a short exact sequence (induced by the fibration)
\[
1\longrightarrow\pi_1(X)\longrightarrow\pi_1(E_h)\overset{p_*}\longrightarrow\Z\longrightarrow1,
\]
where $p_*$ maps $\pi_1(X)$ trivially and $p_*(t)=1$. In other words, $\pi_1(E_h)$ is the semi-direct product
$$\pi_1(E_h)=\pi_1(X)\rtimes_{h_*}\mathbb{Z}.$$

At the purely group theoretic level, the same construction applies as a special case of HNN extensions (as can be done for a pair of topological spaces): If $K$ is a group and $\theta\colon K\to K$ is an automorphism, then we call
\[
\Gamma_\theta=K*_\theta=K\rtimes_\theta\Z:=\langle K,t \ | \ t\alpha t^{-1}=\theta(\alpha), \ \forall\alpha\in K\rangle
\]
the {\em mapping torus} of $K$ with respect to $\theta$. If $K$ is finitely generated, then we say that $\Gamma_\theta$ {\em fibers}. In other words, the 
short exact sequence
\[
1\longrightarrow K\longrightarrow\Gamma_\theta\overset{\pi}\longrightarrow\Z\longrightarrow1,
\]
has finitely generated kernel $K$. As above, $\pi\colon\Gamma_\theta\to\Z$ denotes the projection to the chosen generator of $\Z$. Moreover, we can naturally identify $\Z$ with the quotient group $\Gamma_\theta/K$.

We gather some elementary observations in the following lemma.

\begin{lem}\label{l:observations}
Let $\Gamma_\theta=
K\rtimes_\theta\langle t\rangle$ be a mapping torus. The following hold:
\begin{itemize}
\item[(1)] Every element $x\in\Gamma_\theta$ has the form $x=\beta t^m$, for some $\beta\in K$ and $m\in\Z$.
\item[(2)] If $\Delta$ is a subgroup of $\Gamma_\theta$ of finite index, then $\Delta\cong(K\cap\Delta)\rtimes_{\eta}\Z,$ where $\eta=\theta^{l}|_{K\cap\Delta}$ for some $l\neq0$. If, moreover, $K$ is finitely generated, then $K\cap\Delta$ is finitely generated and so $\Delta$ fibers.
\item[(3)] More generally, if $\Delta$ is a subgroup of $\Gamma_\theta$ such that $\Delta\cap\langle t\rangle\neq1$, then $\Delta\cong(K\cap\Delta)\rtimes_{\eta}\Z$, where $\eta=\theta^{l}|_{K\cap\Delta}$ for some $l\neq0$.
\end{itemize}
\end{lem}
\begin{proof} \
\begin{itemize}
\item[(1)] This follows immediately by the relation $t\alpha t^{-1}=\theta(\alpha)$ for all $\alpha\in K$.
\item[(2)] Since $[\Gamma_\theta:\Delta]<\infty$, we deduce that $\pi(\Delta)$ has finite index, say $l$, in $\Gamma_\theta/K$. The kernel of the composition
\[
\Delta\longhookrightarrow\Gamma_\theta\overset{\pi}\longrightarrow\Z
\]
is equal to $K\cap\Delta$, which is a finite index subgroup of $K$ because 
\[
[K:K\cap\Delta]\leq[\Gamma_\theta:\Delta]<\infty.
\]
Thus, if $K$ is finitely generated, then $K\cap\Delta$ is finitely generated. The above argument is given in the following commutative diagram.
\begin{equation*}\label{eq:subgroup}
\xymatrix{
 & K\cap\Delta\ar[r]^{}\ar[d]^{}&\Delta\ar[d]^{}\ar[r]& \Z\cong\pi(\Delta)\cap\Z \ar[d]^{}
\\ 
 & K\ar[r]^{}&\Gamma_\theta \ar[r]^{\pi}& \Z 
\\ }
\end{equation*}
\item[(3)] Since $\Delta\cap\langle t\rangle\neq1$, the proof is similar to (2), without the assumption on the index.
\end{itemize}
\end{proof}

\subsection{Hirshon's generalisation of Mal'cev's theorem}

\begin{defn}
A group $\Gamma$ is said to be {\em Hopfian} (or to have the {\em Hopf property}) if every surjective endomorphism of $\Gamma$ is an automorphism. It is called {\em residually finite} if for any non-trivial element $\gamma\in\Gamma$ there is a finite group $G$ and a homomorphism $\varrho\colon\Gamma\to G$ such that $\varrho(\gamma)\neq1$. Equivalently, $\Gamma$ is residually finite if the intersection of all its (normal) subgroups of finite index is trivial.
\end{defn}

A well-known theorem of Mal'cev~\cite{Mal} establishes the following connection between the two properties. 

\begin{thm}[Mal'cev]\label{t:Malcev}
Every finitely generated residually finite group is Hopfian.
\end{thm}

Following a similar line of ideas, Hirshon~\cite{Hir} proved the following interesting generalisation of Theorem \ref{t:Malcev} (we state only the torsion-free case). 

\begin{thm}[Hirshon]\label{t:Hirshon}
Let $\Gamma$ be a finitely generated, torsion-free, residually finite group. If $\varphi$ is an endomorphism of $\Gamma$ such that $[\Gamma:\varphi(\Gamma)]<\infty$, then $\varphi$ is injective.
\end{thm}

Since finitely generated residually finite-by-cyclic groups (more generally, split extensions of finitely generated residually finite-by-residually finite groups) are residually finite~\cite[Ch. III, Theorem 7]{Mil}, we derive the following corollary.

\begin{cor}\label{c:residuallyfiniteinjection}
If $K$ is a finitely generated, torsion-free, residually finite group, then every endomorphism of $\Gamma_\theta=K\rtimes_\theta\Z$ onto a finite index subgroup is injective.
\end{cor}
\begin{proof}
Clearly $\Gamma_\theta$ is torsion-free and the corollary follows by Theorem \ref{t:Hirshon}.
\end{proof}

\subsection{Center and virtually trivial fibrations}

The following lemma characterises $\Gamma_\theta$ when the center $C(K)$ of $K$ is trivial. Note that $K$ need not be residually finite.

\begin{lem}\label{l:center}
For a mapping torus $\Gamma_\theta=K\rtimes_\theta\Z$, where  $C(K)=1$, the following conditions are equivalent:
\begin{itemize}
\item[(i)] $C(\Gamma_\theta)\neq1$;
\item[(ii)] there exists $m\neq 0$ such that $\theta^m\in\mathrm{Inn}(K)$;
\item[(iii)] $\Gamma_\theta$ is virtually $K\times\Z$.
\end{itemize}
\end{lem}
\begin{proof}
(i)$\Rightarrow$(ii) Let $x\in C(\Gamma_\theta)$. By Lemma \ref{l:observations}(1) and the assumption that $K$ has trivial center, $x=\beta t^m$ for some $\beta\in K$ and some non-zero integer $m$. Then $(\beta t^m)\alpha=\alpha(\beta t^m)$ for all $\alpha\in K$, that is, $\theta^m(\alpha)=\beta^{-1}\alpha\beta$. Hence, $\theta^m\in\mathrm{Inn}(K)$.

(ii)$\Rightarrow$(iii) Suppose $\theta^m\in\mathrm{Inn}(K)$ for some $m\neq0$, i.e. there is a $\beta\in K$ such that $\theta^m(\alpha)=\beta^{-1}\alpha\beta$ for all $\alpha\in K$. Let  $\Gamma_{\theta^m}$ be the index $m$ subgroup of $\Gamma_\theta$ given by $K\rtimes_{\theta^m}\Z$. A presentation for $\Gamma_{\theta^m}$ is given by
\[
\Gamma_{\theta^m}=\langle K,s \ | \ s\alpha s^{-1}=\theta^m(\alpha)=t^m\alpha t^{-m}, \ \forall\alpha\in K\rangle.
\]
We define 
\begin{equation*}\begin{array}{rcll}\label{eq:iso}
\Psi\colon\Gamma_{\theta^m}&\overset{}\longrightarrow & K\times\Z=K\times\langle z\rangle&\\
 \alpha &\longmapsto & \alpha, \hspace{8pt} 
 \alpha\in K, &\\
  s&\longmapsto & \beta^{-1} z. &
\end{array}\end{equation*}
$\Psi$ is well-defined because
\[
\Psi(s\alpha s^{-1})=
\beta^{-1} z\alpha z^{-1}\beta=\beta^{-1}\alpha\beta=\theta^m(\alpha)=\Psi(\theta^m(\alpha)).
\]
Clearly $\Psi$ is surjective. Finally, $\Psi$ is injective: If $y=\alpha s^n\in\ker(\Psi)$, then
\[
1=\Psi(\alpha s^n)=\alpha\beta^{-n}z^n.
\]
Thus $z^n=\beta^n\alpha^{-1}\in\langle z\rangle\cap K =1$, which implies that $n=0$. We conclude that $\alpha=1$, and so $y=1$. Hence, $\Psi$ is an isomorphism.

(iii)$\Rightarrow$(i) Finally, suppose $\Gamma_\theta$ contains a subgroup $\Delta=K\times\Z$ such that $[\Gamma_\theta:\Delta]=d<\infty$. By Lemma \ref{l:observations}(2), $\Delta\cong(K\cap\Delta)\rtimes_\eta\Z$, where $\eta=\theta^l|_{K\cap\Delta}$ for some $l\geq1$. Since $K\subset\Delta$, we deduce that
\[
l=[\Z:\pi(\Delta)]=[\Gamma_\theta/K:(\Delta K)/K]=[\Gamma_\theta:\Delta K]=[\Gamma_\theta:\Delta]=d.
\]
Hence, $\Delta
\cong K\rtimes_{\theta^d} \Z=\Gamma_{\theta^d}$ (and $C(\Gamma_{\theta^d})=\Z$). In particular, up to multiples, $\theta^{d}\in\mathrm{Inn}(K)$, that is, there is a (possibly trivial) $\beta\in K$ such that $t^{d}\alpha t^{-d}=\theta^{d}(\alpha)=\beta^{-1}\alpha\beta$, for all $\alpha\in K$.
Now, we observe
 \begin{equation*}\label{eq.conjugation}
 \begin{aligned}
\beta^{-1}\alpha\beta  &= \theta^{d}(\alpha)\\
                                    & =  \theta(\theta^{d}(\theta^{-1}(\alpha)))\\
                                    & =  \theta(\beta^{-1}\theta^{-1}(\alpha)\beta)\\
                                    & =  \theta(\beta^{-1})\alpha\theta(\beta), \forall\alpha\in K.
\end{aligned}
\end{equation*}
In particular, $\theta(\beta)\beta^{-1}\in C(K)=1$. We conclude that $\theta(\beta)=\beta$, i.e. $t\beta t^{-1}=\beta$. This implies that $\beta t^{d}\in C(\Gamma_\theta)$.
\end{proof}

\section{Finiteness in outer automorphism and homology groups}

\subsection{Euler characteristic}

The {\em cohomological dimension} of a group $\Gamma$ is defined to be 
\begin{equation}\label{eq:cddef}
 \begin{aligned}
\mathrm{cd}(\Gamma)&
=\inf\{n\in\N \ | \ \Z \ \text{admits a projective resolution of length} \ n\}\\
&=\sup\{n\in\N \ | \ H^n(\Gamma;A)\neq0 \ \text{for some} \ \Gamma\text{-module} \ A\}. 
                                \end{aligned}
\end{equation}
A group $\Gamma$ is said to be of {\em finite homological type} if it has a finite index subgroup  which has finite cohomological dimension and every such subgroup has finitely generated integral homology groups~\cite[Sec. II.3]{Bro1}. A subgroup $\Delta\subseteq\Gamma$ of finite index is of finite homological type if and only if $\Gamma$ is of finite homological type~\cite[Ch. IX, Lemma 6.1]{Bro}.

\begin{defn}\label{d:eulergroup}
Let $\Gamma$ be a group of finite homological type. The {\em Euler characteristic} of $\Gamma$ is defined by
\begin{equation}\label{eq.K}
\chi(\Gamma)= \left\{\begin{array}{ll}
        \sum_i(-1)^i\dim H_i(\Gamma;\R), & \text{if} \ \mathrm{cd}(\Gamma)<\infty,\\
        \\
        \chi(\Delta)/[\Gamma:\Delta], & \text{in general, where} \ [\Gamma:\Delta]<\infty  \ \text{and} \ \mathrm{cd}(\Delta)<\infty.
        \end{array}\right.
\end{equation}
\end{defn}

The above definition, originally due to Brown~\cite{Bro1,Bro}, does not depend on the choice of $\Delta$ as shown by the following theorem.

\begin{thm}{\normalfont(\cite[Ch. IX, Theorem 6.3]{Bro}).}\label{t:eulercoverings}
If $\Gamma$ is a (torsion-free) group of finite homological type and $\Delta$ is a finite index subgroup of $\Gamma$, then $\chi(\Delta)=[\Gamma:\Delta]\chi(\Gamma)$.
\end{thm}

The Euler characteristic is multiplicative under group extensions.

\begin{thm}{\normalfont(\cite[Ch. IX, Prop. 7.3(d)]{Bro}).}\label{t:eulerextensions}
Let $1\to K\to\Gamma\to Q\to1$ be a short exact sequence of groups, where $K$ and $Q$ are of finite homological type. If $\Gamma$ is virtually torsion-free, then it is of finite homological type and $\chi(\Gamma)=\chi(K)\chi(Q)$.
\end{thm}

We immediately obtain the vanishing of the Euler characteristic for mapping tori.

\begin{cor}\label{c:eulermappingtori}
Let $K$ be a torsion-free group of finite homological type and $\Gamma_\theta=K\rtimes_\theta\Z$ be the mapping torus of an automorphism $\theta\colon K\to K$. Then $\chi(\Gamma_\theta)=0$.
\end{cor}

\begin{rem}
Corollary \ref{c:eulermappingtori} tells us that the Euler characteristic of $\Gamma_\theta$ is not applicable in showing that $\Gamma_\theta$ is co-Hopfian.  Similarly, the $L^2$-Betti numbers vanish for $\Gamma_\theta$~\cite{Lue4}.
\end{rem}

We remark that if $\Gamma$ is of finite type, then Definition \ref{d:eulergroup} coincides with Wall's definition via classifying spaces~\cite{Wal1}, and we write $\chi(\Gamma)=\chi(B\Gamma)$. 

\subsection{The finite co-Hopf condition}

\begin{defn}
A group $\Gamma$ is called {\em finitely co-Hopfian} if every injective endomorphism of $\Gamma$ onto a finite index subgroup is an automorphism.
\end{defn}

We will now use condition (\ref{finiteness}) to build a bridge between Lemma \ref{l:center} and the finite co-Hopf property. 

\begin{thm}\label{t:center-finitecoHopf}
Let  $K$ be a group of finite homological type with $\chi(K)\neq0$ and $\theta\colon K\to K$ be an automorphism. If there is no integer $d>1$ such that $\theta^d$ is conjugate to $\theta$ in $\mathrm{Out}(K)$, then the mapping torus $\Gamma_\theta=K\rtimes_\theta\Z$ is finitely co-Hopfian. 

In particular, if $K$ satisfies condition {\normalfont(\ref{finiteness})}, then $\Gamma_\theta$ is finitely co-Hopfian whenever $\theta$ has infinite order in $\mathrm{Out}(K)$.
\end{thm}
\begin{proof}
Let $\varphi\colon\Gamma_\theta\to\Gamma_\theta$ be an injective homomorphism such that $[\Gamma_\theta:\varphi(\Gamma_\theta)]=d<\infty$. By Lemma \ref{l:observations}(2), $\varphi(\Gamma_\theta)$ is a semi-direct product
\begin{equation}
\varphi(\Gamma_\theta)=(K\cap\varphi(\Gamma_\theta))\rtimes_\eta\Z,
\end{equation}
where $\eta=\theta^{l}|_{K\cap\varphi(\Gamma_\theta)}$ for some $l\neq0$. Since
\[
[K:K\cap\varphi(\Gamma_\theta)]\leq[\Gamma_\theta:\varphi(\Gamma_\theta)]<\infty,
\]
and $\chi(K)\neq0$, Theorem \ref{t:eulercoverings} tells us that
\begin{equation}\label{eq:eulerintersection}
\chi(K\cap\varphi(\Gamma_\theta))=[K:K\cap\varphi(\Gamma_\theta)]\chi(K)\neq0.
\end{equation}

Since $\varphi$ is injective, the restriction $\varphi|_K\colon K\to\Gamma_\theta$ is an isomorphism onto its image $\varphi(K)$. Thus
\begin{equation}\label{eq:eulerK}
\chi(K)=\chi(\varphi(K)).
\end{equation}
Moreover, since $\varphi(K)$ is normal in $\varphi(\Gamma_\theta)$ we obtain a commutative diagram
\begin{equation}\label{eq:commutecyclicimage}
\xymatrix{
 & &\Gamma_\theta\ar[d]^{}\ar[r]^{\varphi}&\varphi(\Gamma_\theta) \ar[d]^{}&
\\ 
& & \Gamma_\theta/K\ar[r]^{\overline\varphi \ \ \ }&\varphi(\Gamma_\theta)/\varphi(K), & 
\\ }
\end{equation}
where $\varphi$ and $\overline\varphi$ are isomorphisms onto their images. In particular, $\varphi(\Gamma_\theta)$ is a semi-direct product
\[
\varphi(\Gamma_\theta)\cong\varphi(K)\rtimes\Z\cong K\rtimes\Z.
\]
In addition, $(\varphi|_{\varphi(\Gamma_\theta)})^{-1}(\varphi(K))=\varphi^{-1}(\varphi(K))\cap\varphi(\Gamma_\theta)=K\cap\varphi(\Gamma_\theta)$ (recall $\varphi$ is injective) implies that
\begin{equation}\label{eq:eulerimage}
\chi(\varphi(K))=\chi(K\cap\varphi(\Gamma_\theta))\neq0.
\end{equation}
By (\ref{eq:eulerK}) and (\ref{eq:eulerimage}) we obtain $\chi(K)=\chi(K\cap\varphi(\Gamma_\theta))$ and thus by (\ref{eq:eulerintersection}) we conclude that $K=K\cap\varphi(\Gamma_\theta)$. Hence, Lemma \ref{l:observations}(2) tells us that 
\begin{equation}\label{eq:imagefinal}
\varphi(\Gamma_\theta)=K\rtimes_{\theta^d}\Z=\Gamma_{\theta^d}.
\end{equation}
Repeating the above argument for every $m\in\N$, we obtain a sequence of normal subgroups $\varphi^m(\Gamma_\theta)$ of $\Gamma_\theta$ with
\[
[\Gamma_\theta:\varphi^m(\Gamma_\theta)]=d^m, \ \text{and} \ K\cap\varphi^m(\Gamma_\theta)=K,
\]
which implies that 
\begin{equation}\label{eq:imagefinalm}
\varphi^m(\Gamma_\theta)=K\rtimes_{\theta^{d^m}}\Z, \ \forall m\in\N.
\end{equation}
If $d>1$, we derive
\[
\varphi(K)=\varphi(\cap_m\varphi^m(\Gamma_\theta))\subseteq\cap_m\varphi^{m+1}(\Gamma_\theta)=K.
\]
Hence by (\ref{eq:eulerK}), the injectivity of $\varphi$ and $[\Gamma_\theta:\varphi(\Gamma_\theta)]<\infty$, we deduce that $K=\varphi(K)$; in particular
\[
\varphi|_K\colon K\to K
\]
is an isomorphism.

We have proved that there exists an isomorphism
\begin{equation*}\begin{array}{rcll}\label{eq:iso2}
\varphi\colon\Gamma_{\theta}=K\rtimes_\theta\Z&\overset{}\longrightarrow &
\Gamma_{\theta^d}=K\rtimes_{\theta^d}\Z&\\
 (\alpha,0) &\longmapsto & (\xi(\alpha),0), \hspace{8pt} 
 \xi=\varphi|_K, \ \alpha\in K, &\\
  (e,1)&\longmapsto & (\beta,\pm1),  \hspace{8pt}
  \ \beta\in K,&
\end{array}\end{equation*}
where $e$ denotes the trivial element of $K$. In particular, the relation
\[
(e,1)(\alpha,0)(e,1)^{-1}=(\theta(\alpha),0)
\]
in $\Gamma_\theta$ must be preserved by $\varphi$. We compute the image on both sides to find
\[
\varphi((e,1)(\alpha,0)(e,1)^{-1})=(\beta\theta^{\pm d}(\xi(\alpha))\beta^{-1},0) \ \text{and} \ \varphi(\theta(\alpha),0)=(\xi(\theta(\alpha)),0).
\]
Since $d>1$, the theorem follows.
\end{proof}

\begin{rem}[Finite hyper-Hopf condition]
The above proof (up to the isomorphism $\Gamma_\theta\cong\Gamma_{\theta^d}$) implies as well that, if $\Gamma_\theta$ is {\em finitely hyper-Hopfian} (see Definition \ref{d:hyperHopf}), then it is finitely co-Hopfian. We will discuss the (finite) hyper-Hopf condition in Section \ref{ss:hyper-Hopf}. 
\end{rem}

\subsection{The $K$-module $D=H^n(K;\Z K)$}\label{ss:criterion}

\begin{defn}\label{d:PD}
A group $K$ is said to be a {\em duality group} of dimension $n$ over $\Z$ if
\begin{itemize}
\item[(i)] $\Z$ admits a projective resolution  of finite type and finite length over $\Z K$,  
i.e., $K$ is of type $FP$, and
\item[(ii)] $H^i(K,\Z K)=0$ for all $i\neq n$ and $H^n(K,\Z K)$ is torsion-free as an Abelian group. 
\end{itemize}
The $K$-module $H^n(K,\Z K)$ is called the {\em dualizing module} of $K$. 
If, in addition, $H^n(K,\Z K)\cong\Z$, then $K$ is called {\em Poincar\'e Duality group} or {\em $PD^n$-group} in short. 

Hereafter, the $K$-module $H^n(K,\Z K)$ will be denoted by $D$.
\end{defn}

The above definition, due to Johnson and Wall~\cite{JW}, coincides with the definition given by Bieri~\cite{Bie1,Bie}: $K$ is a duality group of dimension $n$ over $\Z$ if there are natural cap product isomorphisms
\[
H^i(K;A)\overset{\cong}\longrightarrow H_{n-i}(K;D\otimes A),
\]
for all $i$ and all $K$-modules $A$. When $D$ is infinite cyclic, then $K$ is a $PD^n$-group. Note that $\mathrm{cd}(K)=n$.

For any group $K$, the cohomology groups $H^*(K;\Z K)$ inherit a canonical (right) $K$-module structure, and  for any (left) $K$-module $A$ we can define a canonical map form the tensor product $H^*(K;\Z K)\otimes_{\Z K}A$ 
to $H^*(K;A)$. When $K$ is 
of type $FP$ with 
$\mathrm{cd}(K)=n$, then there is 
a canonical isomorphism 
\begin{equation}\label{eq.uctD}
D\otimes_{\Z K}A\overset{\cong}\longrightarrow H^n(K;A).
\end{equation}
Similarly, there is an isomorphism
 \begin{equation}\label{eq.uctD2}
H_n(K;A)\overset{\cong}\longrightarrow\mathrm{Hom}_{\Z K}(D,A); 
\end{equation}
see~\cite[Lemma 9.1]{Bie} for details.

The following remarkable theorem of Strebel~\cite{Str} is particularly useful when dealing with aspherical manifolds:

\begin{thm}{\normalfont(\cite[Theorem, p. 317]{Str}).}\label{t:Strebel}
An infinite-index subgroup of a $PD^n$-group has cohomological dimension at most $n-1$. 
\end{thm}

Bieri~\cite[Prop. 9.22]{Bie} shows an analogous result using {\em homological dimension} $\mathrm{hd}$ (note that $\mathrm{cd}=\mathrm{hd}$ for groups of type $FP$). 

\subsection{The co-Hopf condition}\label{ss:co-Hopf}

\begin{defn}
A group $\Gamma$ is called {\em co-Hopfian} if every injective endomorphism of $\Gamma$ is an isomorphism.
\end{defn}

If $H_n(K;\R)\neq0$, then the universal coefficient type isomorphism, given by (\ref{eq.uctD2}), tells us that $D$ contains a $\Z$-factor. This prompt us to recast the proof of a result of Fel'dman~\cite{Fel} for the cohomological dimensions of short exact sequences.

\begin{thm}[Fel'dman's equality]\label{t:Feldman}
Let $1\to K\to\Gamma\to Q\to1$ be a short exact sequence of groups such that $K$ is of type $FP$ and $\Z\subseteq D=H^n(K;\Z K)$, where $\mathrm{cd}(K)=n$. If $\mathrm{cd}(Q)=q<\infty$, then $\mathrm{cd}(\Gamma)=\mathrm{cd}(K)+\mathrm{cd}(Q)$.
\end{thm}
\begin{proof}
By the Hochschild-Serre spectral sequence~\cite{HS}, we obtain
an isomorphism
\begin{equation}\label{eq.LHSiso}
H^q(Q;H^n(K;A))\cong H^{n+q}(\Gamma;A),
\end{equation}
for any $\Gamma$-module $A$, and 
\begin{equation}\label{eq.LHS}
\mathrm{cd}(\Gamma)\leq\mathrm{cd}(K)+\mathrm{cd}(Q)=n+q.
\end{equation}
Since $\mathrm{cd}(Q)=q<\infty$, we can moreover assume that $A=B\otimes\Z\Gamma$ for some free $\Z$-module $B$, where $H^q(Q;A)\neq0$. Hence, by (\ref{eq.uctD}) we obtain
\begin{equation}\label{eq.Feldman}
 \begin{aligned}
H^{n}(K;A)  = H^n(K;B\otimes\Z\Gamma)
                               \cong D\otimes_{\Z K}(B\otimes\Z\Gamma)
                                     \cong (D\otimes B)\otimes\Z\Gamma.
 \end{aligned}
\end{equation}
Since $\Z\subseteq D$, we conclude that $A\subseteq H^n(K;A)$. This means, by (\ref{eq.LHSiso}) and $H^q(Q;A)\neq0$, that $H^{n+q}(\Gamma;A)\neq0$ and the theorem follows.
\end{proof}

We are now in position to obtain an even stronger conclusion than that of Theorem \ref{t:center-finitecoHopf}, using finiteness conditions in top cohomology. We prove the following result, which is independent of condition (\ref{finiteness}).

\begin{thm}\label{t:co-Hopf}
Let $\Gamma_\theta$ be the mapping torus of an automorphism $\theta\colon K\to K$, where  $K$ is a 
group of type $FP$ with 
 $\mathrm{cd}(K)=n$, such that  
 $H_n(K;\R)\neq0$ and $D=H^n(K,\Z K)$ is finitely generated. If $\Gamma_\theta$ is finitely co-Hopfian, then it is co-Hopfian.
\end{thm}
\begin{proof}
Let $\varphi\colon\Gamma_\theta\to\Gamma_\theta$ be an injective endomorphism. Then $\Gamma_\theta\cong\varphi(\Gamma_\theta)$ and so $\mathrm{cd}(\Gamma_\theta)=\mathrm{cd}(\varphi(\Gamma_\theta))$. Since $K$ is of type $FP$ and $\Z\subseteq D$ (by (\ref{eq.uctD2}) and $H_n(K;\R)\neq0$), Theorem \ref{t:Feldman} yields
\begin{equation}\label{eq.cdequality}
\mathrm{cd}(\varphi(\Gamma_\theta))=\mathrm{cd}(\Gamma_\theta)=\mathrm{cd}(K)+\mathrm{cd}(\Z)=n+1.
\end{equation}
Since $\Gamma_\theta$ is of type $FP$ (because $K$ and $\Z$ are of type $FP$), we conclude that $\varphi(\Gamma_\theta)$ is of type $FP$. 
Now, $\Z\Gamma_\theta\otimes_{\Z\varphi(\Gamma_\theta)}\Z$ has a $\Gamma_\theta$-module structure, and there is a canonical isomorphism (cf. Shapiro's lemma~\cite[Ch. III, Prop. 6.2]{Bro})
\begin{equation}\label{eq.Fh}
H_{n+1}(\varphi(\Gamma_\theta);\Z)\cong H_{n+1}(\Gamma_\theta;\Z\Gamma_\theta\otimes_{\Z\varphi(\Gamma_\theta)}\Z).
\end{equation}
 Moreover, since $D$ is finitely generated and $\Z$ is a $PD^1$-group, we obtain by (\ref{eq.LHSiso}) and (\ref{eq.uctD}) that $\widehat D=H^{n+1}(\Gamma_\theta,\Z\Gamma_\theta)$ is finitely generated as an Abelian group. 
Hence, combining (\ref{eq.Fh}) with (\ref{eq.uctD2}), we obtain
\begin{equation}\label{eq.indexhomology2}
 \begin{aligned}
0\neq H_{n+1}(\varphi(\Gamma_\theta);\Z) &\cong H_{n+1}(\Gamma_\theta;\Z\Gamma_\theta\otimes_{\Z\varphi(\Gamma_\theta)}\Z)\\
                                    & \cong \mathrm{Hom}_{\Z\Gamma_\theta}(\widehat D,\Z\Gamma_\theta\otimes_{\Z\varphi(\Gamma_\theta)}\Z) \\
                                    & \cong\prod_{\rank{\widehat D}} \mathrm{Hom}_{\Z\Gamma_\theta}(\Z,\Z\Gamma_\theta\otimes_{\Z\varphi(\Gamma_\theta)}\Z)\\
                                     & =\prod_{\rank{\widehat D}} (\Z\Gamma_\theta\otimes_{\Z\varphi(\Gamma_\theta)}\Z)^{\Gamma_\theta}.
\end{aligned}
\end{equation}
If $[\Gamma_\theta:\varphi(\Gamma_\theta)]=\infty$, then the last term vanishes (see~\cite{Bie,Str} or~\cite[Ch. III.5]{Bro}). This contradiction shows that $[\Gamma:\varphi(\Gamma_\theta)]<\infty$. Since $\Gamma_\theta$ is finitely co-Hopfian, we derive that $\varphi$ is an isomorphism. 
\end{proof}

\begin{rem}
As a special case of the above, if $K$ is a $PD^n$-group, then $\Gamma_\theta$ is a $PD^{n+1}$-group, and the proof of Theorem \ref{t:co-Hopf} is an immediate consequence of Theorem \ref{t:Strebel}.
\end{rem}

 \section{Cofinitely Hopfian groups (Proof of Theorem \ref{t:maingroups})}

\begin{defn}
A group $\Gamma$ is called
{\em cofinitely Hopfian} if every endomorphism of $\Gamma$ onto a finite index subgroup is an automorphism.
\end{defn}

We are now ready to bring all pieces together to finish the proof of Theorem \ref{t:maingroups}. Our main result is the following characterisation, which encompasses Theorem \ref{t:maingroups}.

\begin{thm}\label{t:Hopf-type}
Suppose $K$ is a 
residually finite group of type $FP$, such that $\chi(K)\neq0$, $H_n(K;\R)\neq 0$ and $D=H^n(K,\Z K)$ is finitely generated, where $\mathrm{cd}(K)=n$. If $K$ satisfies condition {\normalfont(\ref{finiteness})}, then the following are equivalent for any mapping torus $\Gamma_\theta=K\rtimes_\theta\Z$.
\begin{itemize}
\item[(i)] $\Gamma_\theta$ is cofinitely Hopfian;
\item[(ii)] $\Gamma_\theta$ is finitely co-Hopfian;
\item[(iii)] $\Gamma_\theta$ is co-Hopfian;
\item[(iv)] $C(\Gamma_\theta)=1$;
\item[(v)] $\theta$ has infinite order in $\mathrm{Out}(K)$;
\item[(vi)]  $\Gamma_\theta$ is not virtually $K\times\Z$.
\end{itemize}
\end{thm}
\begin{proof}

(i)$\Rightarrow$(ii) This follows by definition. 

(ii)$\Rightarrow$(i) Suppose $\varphi\colon\Gamma_\theta\to\Gamma_\theta$ is a homomorphism such that $[\Gamma_\theta:\varphi(\Gamma_\theta)]<\infty$. Since $K$ is residually finite, Corollary \ref{c:residuallyfiniteinjection} implies that $\varphi$ is injective and the conclusion follows.

(ii)$\Rightarrow$(iii) This is Theorem \ref{t:co-Hopf}, using the finiteness cohomological assumptions.

(iii)$\Rightarrow$(ii) This follows by definition. 

(ii)$\Rightarrow$(v) Suppose 
$\theta^m\in\mathrm{Inn}(K)$ for some $m>0$, i.e., there exists $\beta\in K$ such that $\theta^m(\alpha)=\beta^{-1}\alpha\beta$, for all $\alpha\in K$. Then 
$\Gamma_\theta$ is isomorphic to $\Gamma_{\theta^{m+1}}$. More precisely, there is an injective homomorphism
\begin{equation}\begin{array}{rcll}\label{eq:iso}
\varphi\colon\Gamma_\theta = K\rtimes_\theta\langle t\rangle&\overset{}\longrightarrow & K\rtimes_\theta\langle t\rangle&\\
 \alpha &\longmapsto & \alpha, \hspace{8pt} 
 \alpha\in K, &\\
  t&\longmapsto & \beta t^{m+1}. &
\end{array}\end{equation}
such that $\varphi(\Gamma_\theta)=\Gamma_{\theta^{m+1}}$ and $[\Gamma_\theta:\varphi(\Gamma_\theta)]=m+1>1$. Thus $\Gamma_\theta$ is not finitely co-Hopfian.

(v)$\Rightarrow$(ii) This is Theorem \ref{t:center-finitecoHopf}, using condition (\ref{finiteness}).

(iv)$\Leftrightarrow$(v)$\Leftrightarrow$(vi) This is Lemma \ref{l:center}, under the only assumption that $C(K)=1$.
\end{proof}

\section{The hyper-Hopf property and surface-by-cyclic groups}\label{ss:hyper-Hopf}

We will now extend Theorem \ref{t:Hopf-type} including  an additional equivalent property. Then, we will apply our results to (hyperbolic surface)-by-cyclic groups.

\subsection{The hyper-Hopf condition}

\begin{defn}\label{d:hyperHopf}
A group $\Gamma$ is called {\em hyper-Hopfian} if every endomorphism $\varphi$ of $\Gamma$ such that $\varphi(\Gamma)$ is normal in $\Gamma$ with cyclic quotient $\Gamma/\varphi(\Gamma)$ is an automorphism. It is called {\em finitely hyper-Hopfian} if moreover the quotient $\Gamma/\varphi(\Gamma)$ is required to be finite cyclic.
\end{defn}

A hyper-Hopfian group is clearly finitely hyper-Hopfian. We will see below that the two properties are equivalent for (hyperbolic surface)-by-cyclic groups.

The proof of Theorem \ref{t:center-finitecoHopf} shows that, if $\Gamma_\theta$ is finitely hyper-Hopfian, then it is finitely co-Hopfian. Thus we can extend Theorem \ref{t:Hopf-type} to include the finite hyper-Hopf property.

\begin{thm}\label{t:finite-hyperHopf}
Let $\Gamma_\theta$ be as in Theorem \ref{t:Hopf-type}. $\Gamma_\theta$ is finitely hyper-Hopfian if and only if it satisfies any of the six equivalent conditions of Theorem \ref{t:Hopf-type}.
\end{thm}
\begin{proof}
Any cofinitely Hopfian group is clearly finitely hyper-Hopfian. For the converse, suppose $\Gamma_\theta$ is not finitely co-Hopfian. Let $\varphi\colon\Gamma_\theta\to\Gamma_\theta$ be an injective endomorphism with $1<[\Gamma_\theta:\varphi(\Gamma_\theta)]=d<\infty$. We repeat the steps of the proof of Theorem \ref{t:center-finitecoHopf} to show that $\varphi(\Gamma_\theta)=K\rtimes_{\theta^d}\Z$. Thus $\varphi(\Gamma_\theta)$ is normal and $\Gamma_\theta/\varphi(\Gamma_\theta)\cong\Z_d$. Hence $\Gamma_\theta$ is not finitely hyper-Hopfian.
\end{proof}

\subsection{Surface-by-cyclic groups}\label{ss:surface-by-Z}

Given a finitely generated group $G$, the {\em translation length} of an element $\gamma\in G$, defined by Gersten and Short~\cite{GS}, is given by
\begin{equation}\label{translation}
\tau(\gamma)=\liminf_{n\to\infty}\frac{d(\gamma^n)}{n}\geq0,
\end{equation}
where $d(\cdot)$ denotes the minimal word length with respect to a fixed finite generating set for $G$. The translation length satisfies the following properties~\cite{GS}:
\begin{itemize}
\item $\tau(\gamma)=\tau(\delta\gamma\delta^{-1})$, for all $\gamma,\delta\in G$;
\item $\tau(\gamma^n)=n\tau(\gamma)$, for all $\gamma\in G$, $n\in\N$.
\end{itemize}
Hence, the following result by Farb, Lubotzky and Minsky~\cite{FLM} implies that condition (\ref{finiteness}) is true for aspherical surfaces.

\begin{thm}[Farb-Lubotzky-Minsky]\label{t:FLM}
Every element $\theta$ of infinite order in the mapping class group of a closed aspherical surface has $\tau(\theta)>0$.
\end{thm}

\begin{rem}
\label{r:free-by-cyclic}
An analogous result for infinite order outer automorphisms of a non-elementary free group $F_r$ was proved by Alibegovi\'c~\cite{Ali}. In fact, Alibegovi\'c proved that the translation length of an infinite order outer automorphism $\theta\in\mathrm{Out}(F_r)$ is bounded away from zero, and this is what Bridson, Groves, Hillman and Martin~\cite{BGHM} used in their study of Hopf-type properties for free-by-cyclic groups. Theorem \ref{t:center-finitecoHopf} tells us that the positivity of $\tau(\theta)$ suffices to deduce the finite co-Hopf property for $F_r\rtimes_\theta\Z$. Moreover, observe that the proof  in Theorem \ref{t:finite-hyperHopf}, that finitely hyper-Hopfian implies finitely co-Hopfian, does not use any of the finiteness cohomological conditions on top degree. However, Theorem \ref{t:co-Hopf} does not apply to free-by-cyclic groups, as clearly $D$ is not finitely generated when $K=F_r$. 
Indeed, a finitely co-Hopfian free-by-cyclic group is not necessarily co-Hopfian~\cite[Remark 6.4]{BGHM}. 
Hence, we deduce that all -- but the co-Hopf -- conditions of Theorem \ref{t:finite-hyperHopf} are equivalent for free-by-cyclic groups, which is what was proved in~\cite[Theorem B]{BGHM} and included moreover the hyper-Hopf property. It is therefore natural to ask whether the hyper-Hopf condition is also equivalent to the seven conditions of Theorem \ref{t:finite-hyperHopf}. We will see below that this is true for surface-by-cyclic groups, showing that~\cite[Lemma 6.1]{BGHM} extends over surface-by-cyclic groups. 
\end{rem}

\begin{lem}\label{l:BGHMsurface}
Let $\Sigma_g$ be a closed hyperbolic surface of genus $g$ and  $\Gamma_\theta= \pi_1(\Sigma_g)\rtimes_\theta\Z$ be the mapping torus of an automorphism $\theta\colon\pi_1(\Sigma_g)\to\pi_1(\Sigma_g)$. If $\varphi\colon\Gamma_\theta\to\Gamma_\theta$ is an endomorphism such that $\varphi(\Gamma_\theta)$ is normal in $\Gamma_\theta$ and $\Gamma_\theta/\varphi(\Gamma_\theta)\cong\Z$, then $\Gamma_\theta\cong\pi_1(\Sigma_g)\times\Z$. In particular, $\Gamma_\theta$ is not finitely hyper-Hopfian.
\end{lem}
\begin{proof}
$\Gamma_\theta$ is the fundamental group of an aspherical surface bundle over the circle and so it is a $PD^3$-group. By Scott's theorem for all 3-manifold groups~\cite{Sco} (see also~\cite{Sta}), $\Gamma_\theta$ is coherent, that is, every finitely generated subgroup of $\Gamma_\theta$ is finitely presentable. We deduce that the finitely generated image $\varphi(\Gamma_\theta)$ is finitely presentable. Now we have a short exact sequence
\begin{equation}\label{eq:cd}
1\longrightarrow\varphi(\Gamma_\theta)\longrightarrow\Gamma_\theta\longrightarrow\Z\longrightarrow1,
\end{equation}
where $\Gamma_\theta$ and $\Z$ are Poincar\'e Duality groups and $\varphi(\Gamma_\theta)$ is finitely presentable. Thus by~\cite{Bie,JW} we deduce that $\varphi(\Gamma_\theta)$ is a $PD^2$-group over $\Z$, 
which implies that $\varphi(\Gamma_\theta)$ is a surface group~\cite{Eck}. (Note that to deduce $\mathrm{cd}(\varphi(\Gamma_\theta))=2$ we do not need that $\varphi(\Gamma_\theta)$ is finitely presentable; in particular we do not need that $\Gamma_\theta$ is coherent: Since $\Gamma_\theta$ is a $PD^3$-group and $\varphi(\Gamma_\theta)$ has infinite index in $\Gamma_\theta$, Theorem~\ref{t:Strebel} tells us that $\mathrm{cd}(\varphi(\Gamma_\theta))\leq2$. But $\mathrm{cd}(\Gamma_\theta)\leq\mathrm{cd}(\varphi(\Gamma_\theta))+1$ by the Hochschild-Serre spectral exact sequence for (\ref{eq:cd})~\cite{HS}. Hence, $\mathrm{cd}(\varphi(\Gamma_\theta))=2$.) 
We can thus assume that $\varphi(\Gamma_\theta)\cong\pi_1(\Sigma_h)$, where $\Sigma_h$ is a closed hyperbolic surface of genus $h$. Furthermore, we can assume that the genus $g$ of $\Sigma_g$ is minimal in the decomposition of $\Gamma_\theta$ as a surface-by-cyclic group, so $h\geq g$. In other words,
\begin{equation}\label{eq:surface1}
2h-2
=|\chi(\Sigma_h)|\geq|\chi(\Sigma_g)|
=2g-2.
\end{equation}

Now $\varphi(\pi_1(\Sigma_g))$ is finitely generated normal in the surface group $\pi_1(\Sigma_h)\cong\varphi(\Gamma_\theta)$. A result of Griffiths~\cite{Gri} implies that either $\varphi(\pi_1(\Sigma_g))=1$ or $\varphi(\pi_1(\Sigma_g))$ has finite index in $\varphi(\Gamma_\theta)$. (This can also be deduced from the fact that $\pi_1(\Sigma_h)$ has non-vanishing first $L^2$-Betti number, equal to $|\chi(\Sigma_g)|$;~\cite{Gabo}.) Clearly $\varphi(\pi_1(\Sigma_g))\neq1$, otherwise $\varphi(\Gamma_\theta)$ would be cyclic.  Thus  $\varphi(\pi_1(\Sigma_g))\cong\pi_1(\Sigma_r)$ for some hyperbolic surface of genus $r$. In particular, $r\geq h$, i.e.,
\begin{equation}\label{eq:surface2}
|\chi(\Sigma_r)|\geq|\chi(\Sigma_h)|
\end{equation}
Moreover, we have a homomorphism between surface groups
\[
\varphi|_{\pi_1(\Sigma_g)}\colon\pi_1(\Sigma_g)\to\varphi(\pi_1(\Sigma_g))\cong\pi_1(\Sigma_r).
\]
The Hurewicz map induces a surjective homomorphism $H_1(\Sigma_g)\to H_1(\Sigma_r)$. This means $g\geq r$, i.e.,
\begin{equation}\label{eq:surface3}
|\chi(\Sigma_g)|\geq|\chi(\Sigma_r)|
\end{equation}

 Combining (\ref{eq:surface1}), (\ref{eq:surface2}) and  (\ref{eq:surface3}) we deduce that $\Sigma_g$, $\Sigma_h$ and $\Sigma_r$ are homeomorphic. Since moreover surface groups are Hopfian, we deduce that the restriction
 \[
 \psi:=\varphi|_{\pi_1(\Sigma_g)}\colon\pi_1(\Sigma_g)\to\varphi(\Gamma_\theta)
 \]
is an isomorphism. Hence
\[
(\psi^{-1}\circ\varphi,\pi)\colon\Gamma_\theta\to\pi_1(\Sigma_g)\times\Z
\]
is an isomorphism and the lemma follows.
\end{proof}

We now have the following characterisation for surface-by-cyclic groups.

\begin{thm}\label{t:surface}
Let $K$ be the fundamental group of a closed hyperbolic surface. 
The following are equivalent for any mapping torus $\Gamma_\theta=K\rtimes_\theta\Z$.
\begin{itemize}
\item[(i)] $\Gamma_\theta$ is cofinitely Hopfian;
\item[(ii)] $\Gamma_\theta$ is finitely co-Hopfian;
\item[(iii)] $\Gamma_\theta$ is co-Hopfian;
\item[(iv)] $C(\Gamma_\theta)=1$;
\item[(v)] $\theta$ has infinite order in $\mathrm{Out}(K)$;
\item[(vi)]  $\Gamma_\theta$ is not virtually $K\times\Z$;
\item[(vii)] $\Gamma_\theta$ is finitely hyper-Hopfian;
\item[(viii)] $\Gamma_\theta$ is hyper-Hopfian.
\end{itemize}
\end{thm}
\begin{proof}
Let $K=\pi_1(\Sigma)$, where $\Sigma$ is a closed surface of genus $g>1$. $K$ is residually finite, $PD^2$-group with $\chi(K)\neq0$, and satisfies condition (\ref{finiteness}) by Theorem \ref{t:FLM}.
Theorems \ref{t:Hopf-type}  and  \ref{t:finite-hyperHopf}  imply that the first seven conditions are equivalent. Clearly (viii)$\Rightarrow$(vii), and, by Lemma \ref{l:BGHMsurface}, (vii)$\Rightarrow$(viii).
\end{proof}

\section{$L^2$-Betti numbers and Lefschetz numbers}\label{ss:L2}

\subsection{$L^2$-Betti numbers}

Given a group $\Gamma$ of finite type, the {\em $i$-th $L^2$-Betti number} of $\Gamma$ is defined as the von Neumann dimension of the Hilbert space $\ell^2H^i(B\Gamma)$
 associated to the action of $\Gamma$ on $B\Gamma$. Note that the $L^2$-Betti numbers can be defined for more general group actions~\cite[Definition 6.50]{Luebook}. Nevertheless, since we are dealing with residually finite groups, the reader may refer to L\"uck's  approximation theorem~\cite{Lue1}: If $\Gamma$ is a residually finite group of finite type and $\{\Gamma_k\}_{k\in\N}$ is a residual chain for $\Gamma$, then the $i$-th $L^2$-Betti number of $\Gamma$ is given by
\begin{equation}
b_i^{(2)}(\Gamma)=\lim_{k\to\infty}\frac{\dim H_i(\Gamma_k;\R)}{[\Gamma:\Gamma_k]}.
\end{equation}

By the additivity of the von Neumann dimensions (see~\cite[Theorem 1.35(2)]{Luebook}), we have
\begin{equation}\label{eq.euler-L2}
\chi(\Gamma)=  \sum_i(-1)^ib_i^{(2)}(\Gamma).
\end{equation}

In contrast to the ordinary Betti numbers, $L^2$-Betti numbers are multiplicative under finite index subgroups, i.e., if $\Delta$ is a finite index subgroup of $\Gamma$, then  by~\cite[Theorem 1.35(9)]{Luebook}
\begin{equation}
b_i^{(2)}(\Delta)=[\Gamma:\Delta]b_i^{(2)}(\Gamma).
\end{equation}

Therefore, any argument in the proof of Theorem \ref{t:maingroups} that uses the multiplicativity of the Euler characteristic under finite index subgroups can be carried out if one assumes the non-vanishing of some $L^2$-Betti number of $K$. In addition, the existence of an infinite amenable normal subgroup in $K$ implies the vanishing of all $L^2$-Betti numbers of $K$ (and thus of the Euler characteristic).   By (\ref{eq.euler-L2}) if $\chi(K)\neq0$, then clearly there is some $b_i^{(2)}(K)\neq0$, however the converse is not necessarily true. Singer's conjecture~\cite{Sin} predicts that if $M$ is a closed aspherical manifold of even dimension $2n$, then only $b_{n}^{(2)}(\pi_1(M))$ might not be zero and so $\chi(M)=(-1)^{n} b_{n}^{(2)}(\pi_1(M))$. Thus the conditions $\chi\neq0$ and $b_n^{(2)}\neq0$ might be in fact equivalent for aspherical manifolds. Hence, we have the following  statement which is at least as strong as Theorem \ref{t:finite-hyperHopf}.

\begin{thm}
\label{t:addendum}
Let $K$ be as in Theorem \ref{t:finite-hyperHopf} where instead of $\chi(K)\neq0$, we assume that $b_i^{(2)}(K)\neq0$ for some $i$. Then the following are equivalent for any mapping torus $\Gamma_\theta=K\rtimes_\theta\Z$.
\begin{itemize}
\item[(i)] $\Gamma_\theta$ is cofinitely Hopfian;
\item[(ii)] $\Gamma_\theta$ is finitely co-Hopfian;
\item[(iii)] $\Gamma_\theta$ is co-Hopfian;
\item[(iv)] $C(\Gamma_\theta)=1$;
\item[(v)] $\theta$ has infinite order in $\mathrm{Out}(K)$;
\item[(vi)] $\Gamma_\theta$ is not virtually $K\times\Z$;
\item[(vii)] $\Gamma_\theta$ is finitely hyper-Hopfian.
\end{itemize}
\end{thm}

Nevertheless, we have chosen to use the Euler characteristic instead of the $L^2$-Betti numbers in our statements/proofs, because, assuming the non-vanishing of the Euler characteristic, we were able to find evidence for the validity of condition (\ref{finiteness}) for all aspherical manifolds, which we provide below. Of course, an alternative approach to the one we present below is to show positivity of a translation length-type distance for infinite order outer automorphisms, as in Section \ref{ss:surface-by-Z}.

\subsection{Lefschetz numbers}\label{ss:Lefschetz}

\begin{defn}\label{d:lefschetz}
Let $K$ be a group of finite homological type with $\mathrm{cd}(K)=n$. Given a homomorphism $\theta\colon K\to K$, we define the {\em Lefschetz number} of $\theta$ by
\begin{equation}\label{eq.L}
\Lambda(\theta)=\sum_{i=1}^n (-1)^i\mathrm{tr}(H_i(\theta;\R)),
\end{equation}
where $H_i(\theta;\R)\colon H_i(K;\R)\to H_i(K;\R)$ are the induced homomorphisms and $\mathrm{tr}(\cdot)$ denotes the trace.
\end{defn}

Consider $H_*(K;\R)$ as the vector space $(\oplus_i H_{2i}(K;\R))\oplus(\oplus_i H_{2i+1}(K;\R))$ and the induced by $\theta$ homomorphism  $H_*(\theta;\R)$ as $(\oplus_iH_{2i}(\theta;\R))\oplus(\oplus_i H_{2i+1}(\theta;\R))$. If $\lambda_j$ is an eigenvalue of  $H_*(\theta;\R)$, we define the corresponding generalised eigenspaces by
\[
E_{2i}^{\lambda_j}=\bigcup_{p\geq1}\ker(\oplus_i H_{2i}(\theta;\R)-\lambda_j)^p
\ \text{and} \ E_{2i+1}^{\lambda_j}=\bigcup_{p\geq1}\ker(\oplus_i H_{2i+1}(\theta;\R)-\lambda_j)^p
\]

\begin{defn}
\label{d:l-euler}
The {\em $\lambda_j$-Euler characteristic} of $\theta$ is defined by $\chi_j(\theta)=\dim(E_{2i}^{\lambda_j})-\dim(E_{2i+1}^{\lambda_j})$. We call $\lambda_j$ {\em essential} if $\chi_j(\theta)\neq0$.
\end{defn}

Note that the $\lambda_j$-Euler characteristic does not change under passing to iterates of $\theta$. Thus we can write
\begin{equation}\label{eq.l-euleriterate}
\chi_j=\chi_j(\theta^l), \ \forall l\geq1.
\end{equation}

The following result is of independent interest, and gives a connection between the Euler characteristic of a group $K$ and the Lefschetz number of an iterate of a certain
automorphism of $K$.

\begin{thm}[Theorem \ref{t:finitenessEuler-Lefschetz}]\label{t:Lefschetz}
Let $K$ be a group of finite homological type, $\mathrm{cd}(K)<\infty$, and $\chi(K)\neq0$.
Suppose $\theta\colon K\to K$ is an automorphism which is  
 conjugate to $\theta^d$ in $\mathrm{Out}(K)$ for some $d>1$. 
 Then there exists an integer $q_m=q(\chi(K),d,m)>1$ such that $\chi(K)=\Lambda(\theta^{q_m})$,  for all $m\geq1$.
 \end{thm}
\begin{proof}
Let $\lambda_1,...,\lambda_k$ be the essential eigenvalues of $H_*(\theta)$. Since $\theta$ is an automorphism, $H_*(K;\R)=\oplus_{j=1}^kE_{*}^{\lambda_j}$,
 and $\chi_j$ does not depend on the powers of $\theta$ by (\ref{eq.l-euleriterate}), 
we obtain 
\begin{equation}\label{eq.euler=sumofeigenvauleEuler}
\begin{aligned}
0\neq\chi(K) &=\sum_{i=0}^{\mathrm{cd}(K)}(-1)^i\dim H_i(K;\R)\\
&=\sum_{i \ \text{even}}\dim H_i(K;\R)-\sum_{i \ \text{odd}}\dim H_i(K;\R)\\
&=\sum_{j=1}^k(\dim(E_{2i}^{\lambda_j})-\dim(E_{2i+1}^{\lambda_j}))\\
&=\sum_{j=1}^k\chi_j.
\end{aligned}
\end{equation}
Moreover, using again (\ref{eq.l-euleriterate}), we observe that 
\begin{equation}\label{eq.lefschetziterate}
\Lambda(\theta^l)=\sum_{j=1}^k\chi_j\lambda_j^l, \ \forall l\geq1.
\end{equation}
Suppose now that there is a $\xi\in\mathrm{Aut}(K)$ and $d>1$ such that $\theta^d=\xi\theta\xi^{-1}$ in $\mathrm{Out}(K)$. Thus $\theta^{\mu d^m}=\xi^m\theta^\mu\xi^{-m}$  in $\mathrm{Out}(K)$,  and so $\Lambda(\theta^{\mu d^m})=\Lambda(\theta^\mu)$, for all $m\geq1$, $\mu\geq1$. 
Hence, (\ref{eq.lefschetziterate}) implies
\begin{equation*}\label{eq.iteratesl-euler}
\sum_{j=1}^k\chi_j(\lambda_j^{d^m})^\mu=\sum_{j=1}^k\chi_j\lambda_j^\mu, \ \forall m\geq1, \mu\geq 1,
\end{equation*}
i.e., 
\begin{equation}\label{eq.iteratesl-euler2}
 \underbrace{(\lambda_1^{d^m})^\mu+\cdots+(\lambda_1^{d^m})^{\mu}}_{\chi_1}+\cdots+ \underbrace{(\lambda_k^{d^m})^\mu+\cdots+(\lambda_k^{d^m})^{\mu}}_{\chi_k}= \underbrace{\lambda_1^{\mu}+\cdots+\lambda_1^{\mu}}_{\chi_1}+\cdots+ \underbrace{\lambda_k^{\mu}+\cdots+\lambda_k^{\mu}}_{\chi_k}
\end{equation}
for all $m\geq1, \mu\geq 1$. 
We can rewrite (\ref{eq.iteratesl-euler2}) as
\begin{equation}\label{eq.iterates-symmetric}
\delta_{1,1}^\mu+\cdots+\delta_{1,|\chi_1|}^\mu+\cdots+\delta_{k,1}^\mu+\cdots+\delta_{k,|\chi_k|}^\mu=\widetilde\delta_{1,1}^\mu+\cdots+\widetilde\delta_{1,|\chi_1|}^\mu+\cdots+\widetilde\delta_{k,1}^\mu+\cdots+\widetilde\delta_{k,|\chi_k|}^\mu, \ \forall 
\mu\geq 1,
\end{equation}
where for all $j=1,...,k$ 
\[
(\delta_{j,1}=\cdots=\delta_{j,|\chi_j|}=\lambda_j^{d^m} \ \text{and} \ \widetilde\delta_{j,1}=\cdots=\widetilde\delta_{j,|\chi_j|}=\lambda_j), \ \text{if} \ \chi_j>0
\]
or
\[
(\delta_{j,1}=\cdots=\delta_{j,|\chi_j|}=\lambda_j \ \text{and} \ \widetilde\delta_{j,1}=\cdots=\widetilde\delta_{j,|\chi_j|}=\lambda_j^{d^m}), \ \text{if} \ \chi_j<0.
\]

Note that, if we set $\omega_+:=\sum_{\chi_j>0}\chi_j$ and $\omega_-:=\sum_{\chi_j<0}\chi_j$, then
(\ref{eq.euler=sumofeigenvauleEuler}) is written as
\[
\chi(K)=\sum_{j=1}^k\chi_j=\sum_{\chi_j>0}\chi_j+\sum_{\chi_j<0}\chi_j=\omega_++\omega_-.
\]
In other words, the number of summands in (\ref{eq.iterates-symmetric}) depends on $\chi(K)$ and it is given by $\omega=\omega_+-\omega_-$.

Since $N$ numbers $a_1,...,a_N$ are determined up to order by the $\mu$-th power symmetric functions $a_1^\mu+\cdots+a_N^\mu$, $1\leq\mu\leq N$, we deduce that the two sides of (\ref{eq.iterates-symmetric}) contain the same terms. Hence, there exists $q_m=q(d^m,\omega)>1$ such that
$\lambda_j^{q_m}=1$ for all $j=1,...,k$, and all $m\geq1$, and therefore, by (\ref{eq.euler=sumofeigenvauleEuler}) and (\ref{eq.lefschetziterate}), we obtain 
\begin{equation}\label{eq.lefschetziterate=euler}
\Lambda(\theta^{q_m})=\sum_{j=1}^k\chi_j=\chi(K). 
\end{equation}
\end{proof}

We suspect that Theorem \ref{t:Lefschetz} combined with (a variation of) the additional finiteness conditions in cohomology, the Euler characteristic of the fixed point subgroup of conjugates of $\theta^{q_m}$ and the proof of Theorem \ref{t:co-Hopf} should imply finiteness of the order of $\theta$ in $\mathrm{Out}(K)$ (condition \ref{finiteness}).
However, Theorem \ref{t:Lefschetz} alone does not imply that $\theta$ has finite order in $\mathrm{Out}(K)$, as shown in the following example. 

\begin{ex}\label{ex:positiveEuler-infiniteorder1}
Let $G=BS(1,d)=\langle x,y \ | \ yxy^{-1}=x^d\rangle$
be the $(1,d)$-Baumslag-Solitar group  with $d>1$, and $F_2=\langle u,v\rangle$ be the free group on two generators. Let
\[
K=G\ast F_2
\]
be the free product of $G$ and $F_2$.
Then $\mathrm{cd}(K)=2$ and $\chi(K)=-2$. We define an automorphism $\theta\colon K\to K$ by
\[
\theta(x)=x, \ \theta(y)=xyx^{-1}, \ \theta(u)=u,  \ \theta(v)=v,
\]
and an automorphism $\xi\colon K\to K$ by 
\[
\xi(x)=yxy^{-1}, \ \xi(y)=y, \ \xi(u)=u, \ \xi(v)=v.
\]
Then $\theta^d=\xi\theta\xi^{-1}$ and 
$\theta$ 
has infinite order in $\mathrm{Out}(K)$. Note that for all $q>1$,
\[
\mathrm{Fix}(\theta^{q-1})=\mathrm{Fix}(\theta)=\langle x\rangle\ast\langle u,v\rangle=F_3,
\]
and $\chi(\mathrm{Fix}(\theta^{q-1}))=\Lambda(\theta^{q-1})=\chi(K)=-2$, verifying in particular Theorem \ref{t:Lefschetz}. Also, $H_2(K;\R)\cong H^2(K;\R)=0$ and $[K:\mathrm{Fix}(\theta)]=\infty$. We can easily modify this example so that the real homology in top degree is isomorphic to $\R$.
Let 
\[
L=K\ast\Z^2,
\]
where $K$ is as above. 
We define automorphisms $\hat{\theta}$ and $\hat{\xi}$ on $L$ by
\[
 \hat{\theta}|_K=\theta,  \ \hat{\theta}|_{\Z^2}=\mathrm{id} \ \text{and} \ \hat{\xi}|_K=\xi, \  \hat{\xi}|_{\Z^2}=\mathrm{id},
\]
where $\theta$ and $\xi$ are given as above. 
Then $\hat{\theta}^d=\hat\xi\hat\theta\hat\xi^{-1}$ and 
$\hat\theta$ 
has infinite order in $\mathrm{Out}(L)$. Moreover, for all $q>1$,
\[
\mathrm{Fix}(\hat\theta^{q-1})=\mathrm{Fix}(\hat\theta)=\langle x\rangle\ast F_2\ast\Z^2=F_3\ast\Z^2,
\]
and $\chi(\mathrm{Fix}(\hat\theta^{q-1}))=\Lambda(\hat\theta^{q-1})=\chi(L)=-3$, verifying Theorem \ref{t:Lefschetz}. In addition, we have $H_2(L;\R)\cong H^2(L;\R)\cong\R$. However, $H^2(L;\Z L)$ is not finitely generated. 
\end{ex}

\section{Aspherical manifolds (Proof of Theorem \ref{t:mainmanifolds})}

\subsection{Aspherical manifolds}\label{ss:aspherical}

\begin{defn}\label{d:aspherical}
A finite CW-complex $X$ is called {\em aspherical} if all its higher homotopy groups vanish, i.e. $\pi_n(X)=1$ for all $n\geq 2$.
\end{defn}

We refer to~\cite{Lue3} for a survey on aspherical manifolds.

\begin{thm}[Homotopy classification of aspherical spaces]\label{t:aspherical}
Given two connected aspherical CW-complexes $X$ and $Y$,  
there is a bijection from the set of homotopy classes of maps from X to Y to 
the set of group homomorphisms from $\pi_1(X)$ to $\pi_1(Y)$  $\mathrm{mod \ Inn}(\pi_1(Y))$, sending the homotopy type of a map $f\colon X\to Y$ to the class of the induced homomorphism $f_*\colon\pi_1(X)\to\pi_1(Y)$. In particular, $X$ and $Y$ are homotopy equivalent if and only if their fundamental groups are isomorphic.
\end{thm}

\begin{cor}\label{c:degreeinjection}
Let $X$ be a closed oriented aspherical manifold. If $\varphi\colon\pi_1(X)\to\pi_1(X)$ is an injective endomorphism, 
then $f:=B\varphi\colon X\to X$ has degree $\pm[\pi_1(X):\varphi(\pi_1(X))]<\infty$.
\end{cor}
\begin{proof}
Since $\varphi$ is injective, it is an isomorphism onto its image, which implies $\mathrm{cd}(\pi_1(X))=\mathrm{cd}(\varphi(\pi_1(X)))$. By Theorem~\ref{t:Strebel}, we deduce that $$d=[\pi_1(X):\varphi(\pi_1(X))]<\infty.$$ Hence, there is a finite covering $p\colon\widetilde X\to X$ of degree $d$ and a lift $\widetilde f\colon X\to\widetilde X$ such that $f=p\circ \widetilde f$. Since $\varphi$ is injective, we conclude that $\widetilde f_*$ is an isomorphism. Theorem \ref{t:aspherical} implies that $\widetilde f$ is a homotopy equivalence; in particular, it has degree $\pm1$. Thus 
$$\deg(f)=\deg(\widetilde f)\deg(p)=\pm d.$$
\end{proof}

\subsection{Finishing the proof of Theorem \ref{t:mainmanifolds}}

We will now finish the first, main part of Theorem \ref{t:mainmanifolds}.

\begin{thm}\label{t:mainmanifoldshomotopy}
Let $F$ be a closed aspherical manifold $F$ with non-zero Euler characteristic and residually finite fundamental group, and $h$ be a homeomorphism of $F$. If $\pi_1(F)$ satisfies condition {\normalfont(\ref{finiteness})}, then the following are equivalent:
\begin{itemize}
\item[(i)] $C(\pi_1(E_h))=1$;
\item[(ii)] Every endomorphism of $\pi_1(E_h)$ onto a finite index subgroup induces a 
homotopy equivalence on $E_h$;
\item[(iii)] Every injective endomorphism of $\pi_1(E_h)$ induces a homotopy equivalence on $E_h$;
\item[(iv)] Every self-map of $E_h$ of non-zero degree is a homotopy equivalence. 

\end{itemize}
\end{thm}
\begin{proof}

(i)$\Rightarrow$(ii) By Theorem \ref{t:maingroups}, every endomorphism $\varphi$ of $\pi_1(E_h)$ onto a finite index subgroup is an automorphism, and so $B\varphi\colon E_h\to E_h$ is a homotopy equivalence by Theorem \ref{t:aspherical}. 

(ii)$\Rightarrow$(iv) Let $f\colon E_h\to E_h$ be a map of non-zero degree. Then $f_*(\pi_1(E_h))$ has finite index in $\pi_1(E_h)$ and the claim follows.

(iv)$\Rightarrow$(iii) Let $\varphi\colon\pi_1(E_h)\to\pi_1(E_h)$ be an injective endomorphism. By Corollary \ref{c:degreeinjection}, $\deg(B\varphi)=\pm[\pi_1(E_h):\varphi(\pi_1(E_h))]\neq0$. Thus $B\varphi$ is a homotopy equivalence and $\varphi$ an automorphism.

(iii)$\Rightarrow$(i) Since 
$\pi_1(E_h)$ is co-Hopfian, the claim follows by Theorem \ref{t:maingroups}.
\end{proof}

The Borel conjecture asserts that the homeomorphism type of closed aspherical manifolds is determined by their fundamental groups.

\begin{con}[Borel Conjecture]
If $X$ and $Y$ are closed aspherical manifolds 
with isomorphic fundamental groups, then $X$ and $Y$ are homeomorphic. Moreover, any homotopy equivalence $X\to Y$ is homotopic to a homeomorphism.
\end{con}

\begin{defn}
A closed manifold which satisfies the Borel conjecture is called {\em topologically rigid}.
\end{defn}

Bartels and L\"uck~\cite{BL} made prominent contributions to the Borel conjecture, mostly by studying the Farrell-Jones conjecture. Below we quote what is necessary for us.

\begin{defn}\label{d:classC}
Let $\mathcal{C}$ be the smallest class of groups satisfying the following conditions:
\begin{itemize}
\item $Q\in\mathcal{C}$;
\item If $\pi\colon\Gamma\to Q$ is a group homomorphism such that $\pi^{-1}(U)\in\mathcal{C}$ for any virtually cyclic subgroup $U$ of $Q$, then $\Gamma\in\mathcal{C}$.
\end{itemize}
\end{defn}

Since the Borel conjecture holds for the circle, a special case of the results in~\cite{BL} for groups that fiber reads as follows.

\begin{thm}[Bartels-L\"uck~\cite{BL}]\label{t:BL}
If $M$ is a closed aspherical manifold of dimension at least five and $\pi_1(M)\in\mathcal{C}$, then $M$ is topologically rigid. In particular, if $F$ is a topologically rigid aspherical manifold of dimension at least four, then for any homeomorphism $h\colon F\to F$ the mapping torus $E_h$ is topologically rigid.
\end{thm}

Thus we can strengthen the conclusion of Theorem \ref{t:mainmanifoldshomotopy} by replacing homotopy equivalences with maps homotopic to homeomorphisms, completing the proof of Theorem \ref{t:mainmanifolds}.

\begin{thm}\label{t:mainmanifoldshomeomorphism}
Suppose $F$ is as in Theorem \ref{t:mainmanifoldshomotopy}, which is moreover topologically rigid. Then the homotopy equivalences in (ii), (iii) and (iv) of Theorem \ref{t:mainmanifoldshomotopy} 
are homotopic to homeomorphisms.
\end{thm}
\begin{proof}
If $F$ is a surface, then the 3-manifold $E_h$ is Haken and therefore every self-homotopy equivalence of $E_h$ is homotopic to a homeomorphism by Waldhausen~\cite{Wald}. If $F$ has dimension at least four and satisfies the Borel conjecture, then the claim follows by Theorem \ref{t:BL}.
\end{proof}

\section{Applications}\label{s:applications}

\subsection{The three-dimensional model}
\label{ss:3-manifolds}

Nielsen~\cite{Nie} and Thurston~\cite{Thu} classified surface diffeomorphisms.

\begin{thm}[Nielsen, Thurston]\label{t:NT}
Let $\Sigma$ be a closed hyperbolic surface and $h\colon\Sigma\to\Sigma$ be a diffeomorphism. Then, up to isotopy, $h$ is either
\begin{itemize}
\item[(1)] periodic, i.e., there is an $m$ such that $h^m=id$, or
\item[(2)] reducible, i.e.,  there is an $m$ such that $h^m$ fixes an essential (non-contractible) family of circles, but not periodic, or
\item[(3)] pseudo-Anosov, i.e., there is no $m$ such that $h^m$ fixes an essential family of circles.
\end{itemize}
\end{thm}

\begin{rem}
Our terminology for pseudo-Anosov is slightly different from the standard definition.
\end{rem}

The corresponding mapping tori $E_h$ for $h$ as in (1)-(3) of Theorem \ref{t:NT} are as follows, in terms of Thurston's geometrization picture~\cite{Thubook}.
\begin{itemize}
\item[(1)] $E_h$ possesses the geometry $\mathbb{H}^2\times\R$ and it is virtually $S^1\times\Sigma$;
\item[(2)] $E_h$ has non-trivial JSJ decomposition, i.e., it can be decomposed into geometric pieces along tori, where each piece fibers over the circle (see~\cite{Neu}) and it is either Seifert fibered and modeled on $\mathbb{H}^2\times\R$ or hyperbolic (i.e., modeled on $\mathbb{H}^3$);
\item[(3)] $E_h$ is hyperbolic. 
\end{itemize}

We remark that every hyperbolic 3-manifold is virtually of type (3) in Theorem \ref{t:NT} by the resolution of the Thurston Virtually Fibered Conjecture~\cite{Ago,Ago2}.

The simplest class is that of  manifolds in (1). They admit self-maps of degree greater than one, and, moreover, these maps are homotopic to non-trivial coverings; see Wang's paper~\cite{Wan} or Theorems \ref{t:semi-norms} and \ref{t:Hopfdetail} below for a generalization in all dimensions (see also~\cite{NeoHopf}). 
A manifold $E_h$ in (3) has positive simplicial volume~\cite{Gro,Thubook} and thus does not admit self-maps of  absolute degree greater than one. Furthermore, every self-map of $E_h$ of non-zero degree is homotopic to a homeomorphism by Gromov~\cite{Gro} or Waldhausen~\cite{Wald}, and, even more, to an isometry, by Mostow's rigidity~\cite{Mos,Gro}. 
Let now $E_h$ be a manifold in (2). If $E_h$ has a hyperbolic JSJ piece, then again the simplicial volume of $E_h$ is not zero~\cite{Gro,Thubook}, and so every self-map of $E_h$ of non-zero degree is homotopic to a homeomorphism by Waldhausen~\cite{Wald}. However, if $E_h$ is a {\em non-trivial graph 3-manifold}, i.e., it has no hyperbolic JSJ pieces, then the simplicial volume of $E_h$ vanishes. In that case,  Wang~\cite{Wan} exploited the hyperbolic  
orbifold bases of the Seifert pieces to show that $E_h$ does not admit self-maps of degree greater than one, which means that every self-map of $E_h$ of non-zero degree is a homotopy equivalence and hence homotopic to a homeomorphism~\cite{Wald}. Subsequently, Derbez and Wang~\cite{DW} showed that every non-trivial graph 3-manifold has a finite cover which has non-zero Seifert volume (another Gromov-Thurston type semi-norm introduced by Brooks and Goldman~\cite{BG}). 

The following theorem encompasses a uniform treatment of all cases. In particular, it proves Theorems \ref{t:GromovWang} and \ref{t:surfacebundles}.

\begin{thm}
Let $E_h$ be the mapping torus of a diffeomorphism $h$ of a closed hyperbolic surface $\Sigma$. The following are equivalent:
\begin{itemize}
\item[(i)] $E_h$ is hyperbolic or has non-trivial JSJ decomposition;
\item[(ii)] $C(\pi_1(E_h))=1$;
\item[(iii)] Every endomorphism of $\pi_1(E_h)$ onto a finite index subgroup induces a map homotopic to a homeomorphism;
\item[(iv)] Every injective endomorphism of $\pi_1(E_h)$ induces a map homotopic to a homeomorphism;
\item[(v)] $E_h$ does not admit self-maps of absolute degree greater than one.
\end{itemize}
\end{thm}
\begin{proof}
(i)$\Rightarrow$(ii)  If $E_h$ is hyperbolic or it has non-trivial JSJ decomposition, then  $C(\pi_1(E_h))=1$ by the Seifert fiber space conjecture, the final cases of which were resolved by Casson and Jungreis ~\cite{CJ} and by Gabai~\cite{Gab} -- note that if $E_h$ is hyperbolic, then clearly $C(\pi_1(E_h))=1$.

(ii)$\Leftrightarrow$(iii)$\Leftrightarrow$(iv)$\Leftrightarrow$(v)  These equivalences are part of Theorem \ref{t:maingroups}. More precisely, the equivalences for the homotopy type are part of Theorem \ref{t:mainmanifoldshomotopy} and for the homeomorphism type of Theorem \ref{t:mainmanifoldshomeomorphism} by Waldhausen~\cite{Wald}.

(v)$\Rightarrow$(i) If $E_h$ is not  hyperbolic and it has trivial JSJ decomposition, then $h$ is periodic and hence $E_h$ admits self-maps of degree greater than one; see Theorems~\ref{t:semi-norms} and~\ref{t:Hopfdetail} below (and also~\cite{NeoHopf}). 
\end{proof}

\subsection{Gromov-Thurston norms}

Generalising the concept of the Thurston norm in dimension three~\cite{Thu2}, and of the simplicial volume in any dimension~\cite{Gro}, Gromov introduced the notion of functorial semi-norms~\cite[Ch. 5G$_+$]{Gro2} (see also Milnor-Thurston~\cite{MT}).

\begin{defn}
Let $X$ be a topological space. A {\em functorial semi-norm} in degree $k$ homology is a semi-norm
\begin{equation*}\begin{array}{rrcl}\label{eq:seminorm}
I_{X,k}\colon &H_k(X;\R) & \overset{}\longrightarrow & [0,\infty]\\
&  x & \longmapsto & I_{X,k}(x),  
\end{array}\end{equation*}
such that $ I_{X,k}(x)\geq  I_{X,k}(H_k(f)(x))$ for any continuous map $f\colon X\to Y$.
\end{defn}

By definition, it is straightforward that $I_{X,k}$ is a homotopy invariant. The significance of functoriality is reflected on the mapping degree sets.

\begin{defn}\label{d:degrees}
Given two closed oriented manifolds $M$ and $N$ of the same dimension, the {\em set of degrees of maps} from $M$ to $N$ is defined to be 
$$D(M,N)=\{d\in\Z \ | \ \exists \ f\colon M\to N \ \text{with} \ \deg(f)=d\}.$$ For $M=N$, the set $D(M)=D(M,M)$ is called the {\em set of self-mapping degrees} of $M$.
\end{defn}

Using Definition \ref{d:degrees}, one can assign a natural semi-norm on a closed oriented manifold $N$ by setting in degree $\dim(N)$ homology
\begin{equation}\label{eq:dominationnorm}
I_N([M]):=\sup\{|d| \ | \ d\in D(M,N)\}.
\end{equation}
We refer to~\cite{CL} for further details.

The simplicial volume is  the most prominent example of a functorial semi-norm. It does not vanish on hyperbolic manifolds~\cite{Gro} and, more generally, on aspherical manifolds with non-elementary hyperbolic fundamental groups~\cite{Gro1,Min,Min1}. Another example in dimension three, mentioned in Section \ref{ss:3-manifolds}, is given by the Seifert volume. It was introduced by Brooks and Goldman, who showed that $\widetilde{SL_2(\R)}$-manifolds (in particular, non-trivial circle bundles over hyperbolic surfaces), have non-zero Seifert volume~\cite{BG}. Hyperbolic 3-manifolds and 3-manifolds with non-trivial JSJ decomposition  include all mapping tori $E_h$ of hyperbolic surfaces $\Sigma$, such that $h$ has infinite order in $\mathrm{Out}(\pi_1(\Sigma))$, and either the simplicial volume or the Seifert volume (or both) does not (virtually) vanish on their fundamental classes; cf. Section \ref{ss:3-manifolds}. Higher dimensional analogues of the Seifert volume have been developed in~\cite{DLSW} and those semi-norms reflect on hyper-torus bundles over locally symmetric spaces. The corresponding analogue to non-trivial circle bundles over hyperbolic surfaces, i.e., replacing surfaces by any aspherical manifold with hyperbolic fundamental group, was treated by the author in~\cite{NeoHopf}, using the finiteness of the set of degrees of self-maps as source. We obtain an analogous result for all mapping tori $E_h$, proving in particular Theorem \ref{t:seminorms}.

\begin{thm}\label{t:semi-norms}
Let $E_h$ be as in Theorem \ref{t:mainmanifolds}.
If $h_*$ has infinite order in $\mathrm{Out}(\pi_1(E_h))$, then there is a functorial semi-norm  that takes the value one on $E_h$. If $h_*$ has finite order in $\mathrm{Out}(\pi_1(E_h))$, then every functorial semi-norm  can only take the values $0$ or $\infty$ on $E_h$.
\end{thm}
\begin{proof}
It follows from Theorem \ref{t:center-finitecoHopf} that if $h_*$ has infinite order, then $\pi_1(E_h)$ is finitely co-Hopfian. Furthermore, as a consequence of Theorem \ref{t:Hopf-type}, $\pi_1(E_h)$ satisfies the four equivalent conditions from Theorem \ref{t:mainmanifolds}, and thus $D(E_h)\subseteq\{-1,0,1\}$. We conclude that
\[
I_{E_h}([E_h])=1.
\]

Conversely, suppose $h_*$ has finite order in $\mathrm{Out}(\pi_1(F))$, i.e., there exists an $m\geq 1$ such that $h_*^m\in\mathrm{Inn}(\pi_1(F))$. We have seen in the proof of Theorem \ref{t:Hopf-type} that there is an injective endomorphism $\varphi\colon\pi_1(E_h)\to\pi_1(E_h)$, defined by (\ref{eq:iso}), such that $\varphi(\pi_1(E_h))=\pi_1(E_{h^{m+1}})$ and $[\pi_1(E_h):\pi_1(E_{h^{m+1}})]=m+1>1$. By the definition of $\varphi$ (see also Corollary \ref{c:degreeinjection}) 
the map $f:=B\varphi$ is a self-map of $E_h$ of degree $m+1>1$. Thus every functorial semi-norm $I$ on $E_h$ takes the values
\[
I([E_h])=0 \ \text{or} \ I([E_h])=\infty.
\]
\end{proof}

\subsection{A strong version of the Hopf problem}\label{ss:Hopf}

The following long-standing question of Hopf has essentially motivated the concept of Hopfian groups; see the Kirby list~\cite[Problem 5.26]{Kir} and also~\cite{Neum,Hau}. 

\begin{que}[Hopf]\label{q:Hopf}
\label{Hopfprob}
Let $M$ be a closed oriented manifold. Is every self-map $f\colon M\to M$ of degree $\pm 1$ a homotopy equivalence?
\end{que}

Affirmative answers to Hopf's problem are known in certain cases, such as for simply connected manifolds (this follows by Whitehead's theorem) and for manifolds of dimension at most four with Hopfian fundamental groups by a result of Hausmann~\cite{Hau}.   
If $M$ is aspherical and $\pi_1(M)$ is Hopfian, then every self-map $f\colon M\to M$ of degree $\pm 1$ is a homotopy equivalence, because any map of degree $\pm 1$ induces a surjective endomorphism on $\pi_1(M)$.  For instance, Question \ref{q:Hopf} has an affirmative answer for every aspherical manifold with hyperbolic fundamental group. Note that the assumption $\deg(f)\neq0$ alone would suffice to answer Question \ref{q:Hopf} in the affirmative for the latter manifolds in dimensions greater than one. Indeed, every self-map of non-zero degree of an aspherical manifold with non-elementary hyperbolic fundamental group must have degree $\pm1$. This can be shown either algebraically, based on Sela's work on the Hopf and co-Hopf properties of torsion-free hyperbolic groups~\cite{Sel1,Sel2}, or more geometrically, based on Mineyev's work on bounded cohomology~\cite{Min,Min1} combined with Gromov's work on the simplicial volume and bounded cohomology~\cite{Gro1,Gro}; see~\cite{NeoHopf} for a summary of the proofs.

\begin{thm}[\cite{BHM,Sel1,Sel2,Gro1,Gro,Min,Min1}]
Every self-map of non-zero degree of a closed oriented aspherical manifold with non-elementary hyperbolic fundamental group  
is a homotopy equivalence.
\end{thm}

 Clearly, the above theorem does not hold
for all aspherical manifolds, e.g., the circle admits self-maps of any degree. Nevertheless, every self-map of the circle of degree greater than one is homotopic to a non-trivial covering. The same holds for every self-map of degree greater than one of nilpotent manifolds~\cite{Bel} and of certain solvable mapping tori of homeomorphisms of the $n$-dimensional torus~\cite{Wan1,Neoorder}. In addition, every 
non-zero degree self-map of a closed $3$-manifold $M$ is either a homotopy equivalence or homotopic to a covering map, unless the fundamental group of each prime summand of $M$ is finite or cyclic~\cite{Wan}. Recently, I showed that the same is true for every self-map of non-zero degree of a circle bundle over a closed oriented aspherical manifold with hyperbolic fundamental group~\cite{NeoHopf}. We therefore ask whether the following strong version of Question \ref{q:Hopf} holds for all closed aspherical manifolds.

\begin{que}{\normalfont(\cite[Problem 1.2, Strong version of Hopf's problem for aspherical manifolds]{NeoHopf}).}\label{q:Hopfstrong}
\label{Hopfstrong}
Is every non-zero degree self-map of a closed oriented aspherical manifold either a homotopy equivalence or homotopic to a non-trivial covering? 
\end{que}

\begin{rem}
Together with the Borel conjecture, Question \ref{q:Hopfstrong} can be strengthen as follows: {\em Is every non-zero degree self-map of a closed oriented aspherical manifold homotopic either to a homeomorphism or to a non-trivial covering?} The answer is affirmative in all cases mentioned before Question \ref{q:Hopfstrong} by various rigidity results.
\end{rem}

As an application of Theorem \ref{t:mainmanifolds}, we obtain the following affirmative answer to Question \ref{q:Hopfstrong} for the homeomorphism types of mapping tori of topologically rigid aspherical manifolds.

\begin{thm}[Theorem \ref{t:Hopf}]\label{t:Hopfdetail}
Let $E_h$ be as in Theorem \ref{t:mainmanifolds}, where $F$ is topologically rigid. Every self-map of $E_h$ of non-zero degree is homotopic either to a homeomorphism or 
to a non-trivial covering. In the latter case, $E_h$ is virtually trivial. 
\end{thm}
\begin{proof}
Let $f\colon E_h\to E_h$ be a map of non-zero degree. Since $\pi_1(F)$ is residually finite, $\pi_1(E_h)$ is also residually finite~\cite[Ch. III, Theorem 7]{Mil}, and therefore $f_*$ is injective by Hirshon's theorem~\cite{Hir}, because $[\pi_1(E_h):f_*(\pi_1(E_h))]<\infty$; see Corollary \ref{c:residuallyfiniteinjection}. Since $F$ is topologically rigid, $E_h$ is also topologically rigid by Bartels and L\"uck~\cite{BL} in dimensions greater than four (Theorem \ref{t:BL}) and by Waldhausen~\cite{Wald} in dimension three. By the proof of Corollary \ref{c:degreeinjection} we conclude that $f$ is homotopic to a finite covering. Thus $f$ is homotopic to a homeomorphism if $\deg(f)=\pm1$ and to a non-trivial covering if $|\deg(f)|>1$. Theorem \ref{t:mainmanifolds} tells us (without any assumption on $F$ being topologically rigid) that the case $|\deg(f)|>1$ occurs if and only if $E_h$ is virtually $F\times S^1$. 
\end{proof}

\begin{rem}\label{r:Wanggeneral}
In order to show $\pi_1$-injectivity, we have applied Hirshon's theorem, using that the fundamental group of the total space $E_h$ is residually finite. By Hirshon's result and the residual finiteness for 3-manifold groups~\cite{Hem1}, one deduces the $\pi_1$-injectivity for self-maps of non-zero degree of 3-manifolds with torsion-free fundamental groups, showing thus most of the cases of~\cite[Theorem 0.1]{Wan}. 

Nevertheless, residual finiteness for the fundamental group of the fiber $F$ can be used alone to show $\pi_1$-injectivity, and this approach is of independent interest, generalising (and giving a new proof of) Wang's analogous result~\cite{Wan} for Seifert 3-manifolds modeled on the geometry $\mathbb{H}^2\times\R$, i.e., covered by the product of the circle with a hyperbolic surface:
\begin{prop}\label{p:Wang}
If $h_*$ has finite order in $\mathrm{Out}(\pi_1(F))$, then any self-map of non-zero degree of $E_h$ is $\pi_1$-injective. 
\end{prop}
\begin{proof}
Consider the presentation
\[
\pi_1(E_h)=\langle \pi_1(F),t \ | \ t\alpha t^{-1}=h_*(\alpha), \ \forall\alpha\in\pi_1(F)\rangle,
\]
and recall that $C(\pi_1(F))=1$ because $\chi(\pi_1(F))\neq0$~\cite{Got,Luebook}. Since $h_*$ has finite order in $\mathrm{Out}(\pi_1(F))$,  
there exists $m\neq0$ such that $h_*^m\in\mathrm{Inn}(\pi_1(F))$, i.e., $h_*^m(\alpha)=\beta^{-1}\alpha\beta$ for some $\beta\in\pi_1(F)$. As in the proof of Lemma \ref{l:center}, we deduce that $C(\pi_1(E_h))=\langle\beta t^m\rangle\cong\Z$ and $\pi_1(E_{h^m})\cong\Z\times\pi_1(F)$. Then $\pi_1(E_{h})$ fits into the central extension
\[
1\longrightarrow\langle\beta t^m\rangle\longrightarrow\pi_1(E_{h})\longrightarrow Q=\pi_1(E_{h})/\langle\beta t^m\rangle\longrightarrow1
\]
(i.e., $Q$ is an orbifold group) and, by Lemma \ref{l:observations}, $\pi_1(E_{h^m})$ has the presentation
\[
\pi_1(E_{h^m})=\langle \pi_1(F),t^m \ | \ t^m\alpha t^{-m}=h^m_*(\alpha), \ \forall\alpha\in\pi_1(F)\rangle.
\]
In particular, $\pi_1(E_{h^m})$ is normal of index $m$ in $\pi_1(E_{h})$ and $Q$ fits into the short exact sequence
\begin{equation}\label{eq.shortQ}
1\longrightarrow\pi_1(F)\longrightarrow Q\longrightarrow\Z_m\longrightarrow1.
\end{equation}
By the Hochschild-Serre spectral sequence for (\ref{eq.shortQ}), we deduce that
\begin{equation}\label{eq.homologyQ}
H^*(Q;\R)\cong 
H^*(\pi_1(F);\R).
\end{equation}

Let now $f\colon E_h\to E_h$ be a map of non-zero degree and $f_*\colon\pi_1(E_h)\to\pi_1(E_h)$ be the induced homomorphism. Then $f_*(\pi_1(E_h))$ has finite index in $\pi_1(E_h)$. We lift $f$ to $\widetilde f\colon E_h\to\widetilde{E_h}$, where $\widetilde{E_h}$ is the finite covering of $E_h$ corresponding to $f_*(\pi_1(E_h))$. By Lemma \ref{l:observations}, $\pi_1(\widetilde{E_h})$ fits into the central short exact sequence
\[
1\longrightarrow\langle\beta t^m\rangle\cap\pi_1(\widetilde{E_h})\longrightarrow\pi_1(\widetilde{E_{h}})\longrightarrow\widetilde Q\longrightarrow1,
\]
where $\widetilde Q$ has finite index in $Q$. Since $\widetilde f_*$ is surjective, we conclude that $f_*$ maps $\langle\beta t^m\rangle$ into itself, and thus $f_*$ factors through a 
self-homomorphism $\overline f_*\colon Q\to Q$, whose image $\overline f_*(Q)$ has finite index in $Q$. In particular,  by (\ref{eq.homologyQ}), $H^*(\overline f_*(Q);\R)$ is isomorphic to the cohomology of an aspherical manifold $\overline F$.  

By the lifting property for the orbifolds corresponding to $Q$ and $\overline f_*(Q)$, we lift (factor) the homomorphism
  $\overline f_*$ through a group isomorphic to $\overline f_*(Q)$. The corresponding commutative diagram in cohomology with real coefficients is equivalent to the commutative diagram with the manifold cohomology groups $H^*(F;\R)$ and $H^*(\overline F;\R)$. Since  $F$ and $\overline F$ satisfy Poincar\'e Duality, we conclude that $\dim H^*(F;\R)=\dim H^*(\overline F;\R)$. Hence,
\[
\chi(Q)=\chi(\overline f_*(Q)).
\]
Since, moreover, $[Q:\overline f_*(Q)]<\infty$, we obtain $Q=\overline f_*(Q)$. By $[Q:\pi_1(F)]<\infty$ and the fact that $\pi_1(F)$ is residually finite, we deduce that $Q$ is residually finite, and, in particular, Hopfian by Theorem \ref{t:Malcev}, being finitely generated. Hence $\overline f_*$ is an isomorphism.

We have now shown that $f_*$ maps the infinite cyclic center $\langle\beta t^m\rangle$ of $\pi_1(E_h)$ into itself, and  
factors through an isomorphism of $Q$. 
Hence, $f_*$ is injective.
\end{proof}
\end{rem}

\subsection{Gromov's ordering and rigidity}\label{ss:Gromovorder}

In a lecture given at CUNY in 1978, Gromov suggested studying the existence of maps of non-zero degree as defining an ordering on the homotopy types of closed oriented manifolds of a given dimension~\cite{CT}.

\begin{defn}
Let $M$, $N$ be closed oriented manifolds of dimension $n$. If there is a map of non-zero degree $M\to N$, we say that {\em $M$ dominates $N$} and denote this by $M\geq N$.
\end{defn}

For instance, if $\Sigma_g$ and $\Sigma_h$ are two closed surfaces of genus $g,h\geq0$, then the domination relation is a total ordering given by
\begin{equation*}\label{eq.ordersurface}
\Sigma_g\geq\Sigma_h \ \Longleftrightarrow \ g\geq h.
\end{equation*}
The implication ``$\Rightarrow$" is a simple application of the Poincar\'e Duality in all dimensions: If $M\geq N$, then the ordinary Betti numbers satisfy $b_i(M)\geq b_i(N)$, for all $i=1,...,n$. 
The implication ``$\Leftarrow$" is very specific to dimension two; clearly inequalities between Betti numbers do not guarantee domination between two given manifolds. In general, it is a hard problem to show existence of a map of non-zero degree between two given manifolds. Moreover, the domination relation for maps of degree $\pm1$ defines a partial order on the homotopy types of aspherical manifolds with Hopfian fundamental groups. Variations of this partial ordering (for any degree) have been studied in many instances~\cite{Bel,BBM,CT,Ron,Wan1,Neoorder}. 

By Bartels and L\"uck~\cite{BL}, the class $\mathcal{C}$ in Definition \ref{d:classC} contains hyperbolic groups, and Theorem \ref{t:BL} tells us that aspherical manifolds with hyperbolic fundamental groups are topologically rigid in dimensions higher than four; for negatively curved  manifolds this was also proved by Farrell and Jones~\cite{FJ}. The same result is true in dimension three by Perelman's proof of the geometrization conjecture. Suppose now $M$ and $N$ are closed aspherical manifolds with non-elementary hyperbolic fundamental groups and let $f\colon M\to N$ and $g\colon N\to M$ be maps of non-zero degree. Then $g\circ f$ and $f\circ g$ are self-maps of non-zero degree of $M$ and $N$ respectively. As mentioned in Section \ref{ss:Hopf}, every self-map of non-zero degree of an aspherical manifold with non-elementary hyperbolic fundamental group must have degree $\pm1$. 
Thus  $g_*\circ f_*$ and $f_*\circ g_*$ are automorphisms of $\pi_1(M)$ and $\pi_1(N)$ respectively (recall that $\pi_1(M)$ and $\pi_1(N)$ are Hopfian) and this proves the following theorem.

\begin{thm}[\cite{BHM,BL,Sel1,Sel2,Min,Min1,Gro1}]\label{t:rigidityhyperbolic}
Let $M$, $N$ be two closed aspherical manifolds of dimension $n\geq2$ with hyperbolic fundamental groups. If $M\geq N\geq M$, then $M$ and $N$ are homotopy equivalent. If, moreover, $n\neq 4$, then $M$ and $N$ are homeomorphic.
\end{thm}

With the above discussion, it is now straightforward that Theorem \ref{t:mainmanifolds} implies the analogous rigidity result for non-virtually trivial mapping $E_h$. This in particular rediscovers the (homotopy part of the) case of the corresponding negatively curved 3-manifolds.

\begin{thm}[Theorem \ref{t:Gromovorder}]
For $i=1,2$, let $E_{h_i}$ be as in Theorem \ref{t:mainmanifolds} such that  ${h_i}_*$ has infinite order in $\mathrm{Out}(\pi_1(F_i))$. If $E_{h_1}\geq E_{h_2}\geq E_{h_1}$, then $E_{h_1}$ and $E_{h_2}$ are homotopy equivalent. If, moreover, the $F_i$ are topologically rigid, then  $E_{h_1}$ and $E_{h_2}$ are homeomorphic.
\end{thm}

\subsection{The Lichnerowicz problem (quasiregular maps)}\label{ss:quasiregular}

A {\em uniformly quasiregular map} of a $n$-manifold $M$ is a map which is rational with respect to (i.e., preserves) some bounded measurable conformal structure on $M$. Recall that a measurable conformal structure on $M$ is a measurable map that assigns to each element of $M$ a symmetric, positive definite $n\times n$ matrix of determinant one (see~\cite{Tuk}).

The Lichnerowicz conjecture~\cite{Lic}, proved by Lelong-Ferrand~\cite{Lel}, states that the only compact manifold of dimension $n$ that admits a non-compact conformal automorphism group is $S^n$. Lelong-Ferrand's work suggests moreover the following non-injective version of the Lichnerowicz conjecture~\cite[p. 1614]{BHM},~\cite[p. 2092]{MMP}.

\begin{que}[Lichnerowicz problem]
Which compact manifolds admit non-injective uniformly quasiregular maps? 
\end{que}

Martin, Mayer and Peltonen~\cite{MMP} showed that the only manifolds of dimension $n$ which admit locally (but not globally) injective uniformly quasiregular maps are those quasiconformally homeomorphic to the $n$-dimensional Euclidean space forms. Bridson, Hinkkanen and Martin~\cite{BHM} exploited results by Walsh~\cite{Wal}, Smale~\cite{Sma}, V\"ais\"al\"a~\cite{Vai} and \v Chernavski\u{\i}~\cite{Cer}, together with Sela's work on the Hopf and co-Hopf properties~\cite{Sel1,Sel2}, to show that closed manifolds with non-elementary hyperbolic groups do not admit any non-trivial quasiregular maps.

\begin{thm}{\normalfont(Bridson-Hinkannen-Martin~\cite{BHM}).}\label{t:BHM}
Every quasiregular map of a closed manifold with non-elementary torsion-free hyperbolic group is a homeomorphism.
\end{thm}

We will explain how to combine our results with those by Walsh~\cite{Wal}, Smale~\cite{Sma}, V\"ais\"al\"a~\cite{Vai} and \v Chernavski\u{\i}~\cite{Cer} to obtain the analogous theorem to Theorem \ref{t:BHM} 
for manifolds whose fundamental groups fullfil the conditions of Theorem \ref{t:maingroups}. 

Recall that a map is called {\em open} if the image of each open set is open, and {\em proper} if the preimage of each compact set is compact. 

\begin{lem}[Walsh~\cite{Wal}, Smale~\cite{Sma}]\label{l:WalshSmale}
Let $M$, $N$ be two finite CW-complexes. If $f\colon M\to N$ is a proper, open and surjective map, then $f_*(\pi_1(M))$ has finite index in $\pi_1(N)$.
\end{lem}

\begin{rem}\label{r:WS}
Lemma \ref{l:WalshSmale} tells us that the results of our paper apply to open maps between finite CW-complexes whose fundamental groups satisfy the  conditions of Theorem \ref{t:maingroups} and have trivial center.

Note that the converse to Lemma \ref{l:WalshSmale} is true for compact connected PL manifolds $M$ and $N$~\cite[Cor. 5.15.3]{Wal}: If $f\colon M\to N$ satisfies $[\pi_1(N):f_*(\pi_1(N))]<\infty$, then $f$ is homotopic to a light open map, where {\em light} means that the preimage of every point is totally disconnected.
\end{rem}

Any quasiregular map is open and {\em discrete}, i.e., the preimage of any point consists of isolated points. In particular, the preimage of any point under a proper quasiregular map is a finite set. The following theorem was proved independently by \v Chernavski\u{\i}~\cite{Cer} and V\"ais\"al\"a~\cite{Vai}.

\begin{thm}[\v Chernavski\u{\i}, V\"ais\"al\"a]\label{t:branch}
 If $f\colon M\to M$ is an open discrete map of a manifold $M$ of dimension $n$, then $f$ is a local homeomorphism except for a set $B_f$ whose dimension is at most $n-2$. Furthermore, the dimension of $f(B_f)$ is at most $n-2$. 
 \end{thm}

The set $B_f$ is called the {\em branch set} of $f$. 
We are now ready to prove Theorem \ref{t:quasiregular}. We first state a more general result, not necessarily for closed manifolds; compare~\cite[Theorem 5.1 and Remark 5.2]{BHM} and~\cite[Section 7, Remark]{Vai}.

\begin{thm}\label{t:quasiregulargeneral}
Let $M$ be a finite CW-complex with fundamental group $\pi_1(M)=K\rtimes_\theta\Z$, where $K$ is as in Theorem \ref{t:maingroups}. If $C(\pi_1(M))=1$, then every proper, open and surjective map $f\colon M\to M$ induces an automorphism of $\pi_1(M)$.
\end{thm}
\begin{proof}
By Lemma \ref{l:WalshSmale}, $f_*(\pi_1(M))$ has finite index in $\pi_1(M)$. Since the order of $\theta$ in $\mathrm{Out}(K)$ is infinite, Theorem \ref{t:maingroups} tells us that $\pi_1(M)$ is cofinitely Hopfian, and thus $f_*$ is an isomorphism.
\end{proof}

\begin{cor}[Theorem \ref{t:quasiregular}]
Every quasiregular map of a closed manifold $M$ with fundamental group $\pi_1(M)=K\rtimes_\theta\Z$ as in Theorem \ref{t:maingroups}, where $C(\pi_1(M))=1$, is a homeomorphism.
\end{cor}
\begin{proof}
Let $f\colon M\to M$ be a quasiregular map. Then $f$ is proper, open and discrete and thus finite-to-one ($M$ is closed), thus $f_*$ is an automorphism of $\pi_1(M)$ by Theorem \ref{t:quasiregulargeneral}. By~\cite{HW} (see also~\cite{BHM}), the branch set $B_f$ in Theorem \ref{t:branch} does not separate locally $M$ at any point, thus $f$ restricts to a covering on
\[
M\setminus f^{-1}(f(B_f))\longrightarrow M\setminus f(B_f),
\]
and the preimage of each $x\in M\setminus f(B_f)$ under $f$ contains one 
point because $\deg(f)=\pm1$. Thus $B_f=\emptyset$, the map $f$ is a covering, and hence a homeomorphism; see also~\cite[Section 5]{BHM} and~\cite{Vai}.
\end{proof}

\section{Extension to non-aspherical manifolds}\label{s:non-aspherical}

The purely algebraic nature of Theorem \ref{t:maingroups}, and of the various techniques applied in different stages of the proof, suggest that variants of the results of this paper generalise to non-aspherical manifolds. Such an extension was already given in Section \ref{ss:quasiregular} for open maps between finite CW-complexes. Below we give two more illustrative classes of groups. 

\subsection{Connected sums} 

For $s\geq2$, let the connected sum
\[
M=M_1\#\cdots\# M_s, 
\]
where $M_1,...,M_s$ are closed $n$-dimensional manifolds, such that their fundamental groups $K_i=\pi_1(M_i)$ are residually finite and $\chi(\pi_1(M))\neq0$. 

Suppose, for example, that $M_1$ is aspherical with $\chi(M_1)\neq0$, $\pi_1(M_1)$ satisfies condition (\ref{finiteness}), and that for $i=2,...,s$ each $M_i$ is simply connected. Then $\pi_1(M)$ satisfies all assumptions of Theorem \ref{t:maingroups} and thus Theorem \ref{t:maingroups} applies to  the fundamental group of the non-aspherical mapping torus $E_h$ of a homeomorphism $h\colon M\to M$. 

\subsection{Products}

Another example comes from direct products. Let, for instance,
\[
K=F_{r_1}\times\cdots\times F_{r_s}\times\pi_1(\Sigma_{g_1})\times\cdots\times\pi_1(\Sigma_{g_t}),
\]
where each $F_{r_i}$ is free on $r_i>1$ generators and each $\Sigma_{g_j}$ is a closed surface of genus $g_j\geq2$. Then $K$ is residually finite and $\chi(K)\neq0$. As explained already, the free factors $F_{r_i}$ do not satisfy the conditions of Theorem \ref{t:maingroups}, and so the same is true for $K$.

Let $\theta\colon K\to K$ be an automorphism. By~\cite{NeoAnosov}, since $F_{r_i}$ and $\pi_1(\Sigma_{g_j})$ are Hopfian and have trivial center, there exists an integer $m$ such that $$\theta^m=\theta^m|_{F_{r_1}}\times\cdots\times\theta^m|_{F_{r_s}}\times\theta^m|_{\pi_1(\Sigma_{g_1})}\times\cdots\times\theta^m|_{\pi_1(\Sigma_{g_t})},$$
where $\theta^m|_{F_{r_i}}$ and $\theta^m|_{\pi_1(\Sigma_{g_j})}$ denote self-automorphisms. Now we apply Theorem \ref{t:maingroups} to each $\theta^m|_{\pi_1(\Sigma_{g_j})}$ and Remark \ref{r:free-by-cyclic} (or~\cite[Theorem B]{BGHM}) to each $\theta^m|_{F_{r_i}}$ to deduce that Theorem \ref{t:maingroups} (except for the co-Hopf condition) holds for the mapping torus $\Gamma_\theta=K\rtimes_\theta\Z$.

\bibliographystyle{alpha}

\end{document}